\documentclass[11pt,reqno]{amsart}
\usepackage{amsmath}
\usepackage{amssymb}
\usepackage{amsthm}
\usepackage{enumerate}
\usepackage[usenames]{color}
\usepackage{eepic,epic}
\usepackage{url} 
\usepackage{mathtools}
\usepackage{epstopdf}

\usepackage{graphicx}

\newtheorem{thm}{Theorem}[section]

\newtheorem{lem}[thm]{Lemma}
\newtheorem{prop}[thm]{Proposition}

\theoremstyle{definition}
\newtheorem{defn}{Definition}

\theoremstyle{remark}
\newtheorem{rem}{Remark}
\numberwithin{equation}{section}
\numberwithin{rem}{section}
\numberwithin{defn}{section}

\begin{document}
\title[]
{Convergence results of a nested decentralized gradient method for non-strongly convex problems} 
\author{Woocheol Choi}
\address{Department of Mathematics, Sungkyunkwan University, Suwon 440-746, Republic of Korea}
\email{choiwc@skku.edu}

\author{Doheon Kim}
\address{Department of Applied Mathematics, Hanyang University, Hanyangdaehak-ro 55, Sangnok-gu, Ansan, Gyeonggi-do 15588, Republic of Korea}
\email{doheonkim@hanyang.ac.kr}

\author{Seok-Bae Yun}
\address{Department of Mathematics, Sungkyunkwan University, Suwon 440-746, Republic of Korea}
\email{sbyun01@skku.edu}

\subjclass[2010]{Primary 90C25, 68Q25}

\keywords{Distributed Gradient methods, NEAR-DGD$^{+}$, Quasi-strong convexity}

\maketitle

\begin{abstract}
We are concerned with the convergence of NEAR-DGD$^+$ (Nested Exact Alternating Recursion Distributed Gradient Descent) method introduced to solve the distributed optimization problems. 
Under the assumption of the strong convexity of local objective functions and the Lipschitz continuity of their gradients, the linear convergence is established in \cite{BBKW - Near DGD}. In this paper, we investigate the convergence property of NEAR-DGD$^+$ in the absence of strong convexity. More precisely, we establish the convergence results in the following two cases: (1) When only the convexity is assumed on the objective function. (2) When the objective function is represented as a composite function of a strongly convex function and a rank deficient matrix, which falls into the class of  convex and quasi-strongly convex functions. Numerical results are provided to support the convergence results. 
\end{abstract}

\section{Introduction}
Recent years have seen a growing interest in developing algorithms to optimize a system in which several networked agents cooperate to minimize the aggregate cost function:
\begin{align}\label{1.1}
\min_{x\in\mathbb{R}^p} f(x):= \frac{1}{n} \sum_{i=1}^nf_i(x) ,
\end{align}
which often goes by the name of the distributed optimization in the literature.
Here $f_i:\mathbb{R}^p\rightarrow\mathbb{R}$ is a local objective function (or cost function of agent $i$) and $n$ is the number of agents.
Such formulation is relevant in various situations where the resource has to be optimally allocated. Examples include multi-agent system \cite{BCM, CYRC}, wireless sensor networks \cite{QT, BKP, SKKM}, and  machine learning problems \cite{BCN, FCG, RB, YYWY}, to name a few.

The system  $(\ref{1.1})$ is often studied as an equivalent decentralized optimization problem:
\begin{align*}\label{concen form}
\begin{split}
\min_{{\bf x}\in\mathbb{R}^{np}} F({\bf x}) \mbox{ under the constraint }
~ \big(W\otimes I_p\big){\bf x}={\bf x},
\end{split}
\end{align*}
where $F$ denotes the aggregate cost function in (\ref{1.1}):
\[
F({\bf x}) =\sum_{i=1}^nf_i(x_i).
\]
Here $I_p$ is the $p\times p$ identity matrix, ${\bf x}$ denotes an $np$ dimensional column vector made by stacking up $x_i\in\mathbb{R}^p$ $(i=1,\cdots,n)$:
 
 $${\bf x}= (x_{1}^{\top}, x_{2}^{\top}, \cdots, x_{n}^{\top})^{\top},$$
  and the consensus matrix $W$ is the  $n\times n$ matrix that contains the information on the connectivity of the underlying network:
\begin{equation}\label{stack}
W= \begin{pmatrix} w_{11} & w_{12} & \cdots & w_{1n} \\ w_{21} & w_{22} & \cdots &w_{2n} \\ \vdots &\vdots &\ddots &\vdots \\ w_{n1} & w_{n2} &\cdots &w_{nn}
\end{pmatrix}. 
\end{equation}
\textcolor{black}{For a $k\times l$ matrix $A$ and an $m\times n$ matrix $B$, the Kronecker product $A\otimes B$ is defined as the following  $km\times ln$ matrix: 
\begin{align*}
A\otimes B=\left(\begin{array}{ccc}
A_{11}B&\cdots&A_{1l}B\\
\vdots& \ddots&\vdots\\
A_{k1}B&\cdots&A_{kl}B
\end{array}\right)
\end{align*}
where $A_{ij}B$ denotes an $m\times n$ block matrix.}	
Throughout this paper, we assume the following properties for the consensus matrix:
{\color{black}
\begin{itemize}
\item $W$ is a doubly-stochastic $n\times n$ matrix.
\item The directed graph associated to $W$ is strongly connected, and has a self-loop at every vertex.
\end{itemize}
In Lemma \ref{newlemma}, we will show that the graph associated to $W^\top W$ is connected, and so the eigenvalue $1$ of $W^\top W$ is simple and all other eigenvalues have modulus strictly less than 1. We let $\beta\in[0,1)$ be the square root of the second largest eigenvalue of $W^\top W$, which is equal to the spectral norm of $W-\frac{1}{n}1_n1_n^\top$, where $1_n$ is a column vector in $\mathbb{R}^n$ whose entries are all $1$.} We let $f^*$ be the optimal objective value of \eqref{1.1} and denote by ${X}^*$ the set of optimal solutions to \eqref{1.1} which is assumed to be nonempty, i.e., 
\begin{equation*}
{X}^* = \{x \in \mathbb{R}^p ~:~ f(x) =f^*\}.
\end{equation*}
Several approaches have been suggested to solve the distributed optimization problem \eqref{1.1}. Since the
literature is huge, we only review works that are directly related to the decentralized method.
The most common category of algorithms to solve the aforementioned distributed optimization (\ref{1.1}) is
the distributed gradient descent (DGD), which is  a gradient descent method with a weighted averaging operation using the consensus matrix in the non-gradient part. After the introduction of DGD in \cite{Ned DGD} (See also \cite{YLY}), numerous extensions and modifications of this method have been suggested.  For example, extension to the case where noise or randomness is present \cite{LO - random,MB - random,R Ned V - stochastic,S Ned - stochastic}, extension to optimizations on the network represented by an undirect graph \cite{Ned Push sum}, constrained problems \cite{Ned O P - constraint}, and quantization effect \cite{Ned O - Quantized} to name a few. Nice overview of DGD can be found in \cite{GRBDG review,Nedic review,Many review}. For a comparison of the decentralized method and centralized method, we refer to \cite{LZZ comparison}.

A common problem shared by the  family of the distributed gradient methods
is that,  when the step size is fixed \cite{SLWY - Extra}, the iteration does not lead to the exact optimizing solution, but only to a neighborhood of it. The exact optimizing solution can be reached only when a proper diminishing of the step size is accompanied.
Several attempts to overcome this drawback and to obtain the convergence to the exact optimizing solution with fixed time step have been suggested. 
In \cite{SLWY - Extra}, a multi-step decentralized gradient descent is suggested by subtracting the
DGD of one step from the next.
The gradient tracking method was introduced in 
\cite{P Nedic - tracking,LS-NEXT,NOS - Tracking } where an auxiliary variable is introduced to track the difference of the gradients in a consensus setting.  
In NEAR-DGD$^+$ method  \cite{BBKW - Near DGD,BBW NEAR DGD convergence}, increasingly aggregate consensus process is imposed on each step to guarantee the convergence to the exact optimizing solution without reducing the step size. In \cite{BBKW - Near DGD},  the linear convergence was obtained for NEAR-DGD$^+$ when each local cost function $f_i$ is strongly convex and smooth. In the current work, we extend the convergence estimate to the cases of convex functions when such strong convexity is missing.  Specifically, we establish the convergence results \textcolor{black}{in the following two cases: (1) When only the convexity is assumed on the objective function $f$. (2) When the objective function $f$ is represented as a composite function of a strongly convex function and a rank deficient matrix, which falls into the class of  convex and quasi-strongly convex functions.}


 This paper is organized as follows: In Section \ref{sec-2}, we state the main results of this paper and prove two preliminary lemmas concerning the property of cost functions. In Section \ref{sec-3}, we prove a couple of preliminary results used for the convergence analysis. Section \ref{sec-4} is devoted to analyzing the convergence properties of NEAR-DGD$^{+}$ in the case of convex cost functions. In Section \ref{sec-5}, we apply the argument used in Section \ref{sec-4} to obtain a convergence result of NEAR-DGD$^{+}$ for the quasi-strongly convex case. In Section \ref{sec-6}, we develop another argument to establish more sharp convergence results in the quasi-strongly convex case. Section \ref{sec-7} provides numerical tests supporting  the validity of our result.
 
\medskip

Before ending this section, we state the following notation used in the paper.\newline

\noindent {\bf Notation.} $\|\cdot\|$ denotes the standard $\ell^2$-norm in the Euclidean space, or the operator norm induced by the $\ell^2$-norm, depending on the context.

\section{Main results}\label{sec-2}

In this section, we describe the detail of NEAD DGD$^{+}$ and state the main results of this paper. 

\begin{defn}For given $\alpha \geq 0$, a function $\mathbf{f}: \mathbb{R}^p \rightarrow \mathbb{R}$ is called $\alpha$-strongly convex (convex if $\alpha =0$) if
\begin{equation*}
\mathbf{f}(y) \geq \mathbf{f}(x) + \langle y-x, \nabla \mathbf{f}(x) \rangle+ \frac{\alpha}{2} \|y-x\|^2\quad \forall~x,y \in \mathbb{R}^p.
\end{equation*}
\end{defn}
\begin{defn}We  say that a function $\mathbf{f}: \mathbb{R}^p \rightarrow \mathbb{R}$ is called $L$-smooth for given   $L >0$ if $f$ has Lipschitz continuous gradient with constant $L$, i.e., 
\begin{equation*}
\|\nabla \mathbf{f}(x) - \nabla \mathbf{f}(y)\| \leq L \|x-y\|\quad \forall~x,y \in\mathbb{R}^p.
\end{equation*}
\end{defn}

\subsection{NEAR-DGD$^+$}
The Nested Exact Alternating Recursion Distributed Gradient Descent (NEAR-DGD$^{+}$) was introduced in \cite{BBKW - Near DGD,BBW NEAR DGD convergence} to solve (\ref{1.1}): 
\begin{equation}\label{NEAR DGD}
{\bf {\bf x}}_{k+1} = \big(W^{t(k)}\otimes I_p\big)\big({\bf x}_k - \mu \nabla F({\bf x}_k)\big)\quad \forall~ k \geq0,
\end{equation} 
where ${\bf x}_k= (x_{1,k}^{\top}, x_{2,k}^{\top}, \cdots, x_{n,k}^{\top})^{\top}$ and $\mu >0$. Here  $\nabla F (\mathbf{x}_k)\in \mathbb{R}^{np}$ denotes
\begin{equation*}
\nabla F({\bf x}_k)=\left(
\nabla f_1(x_{1,k})^{\top},\cdots
, \nabla f_n(x_{n,k})^{\top}
\right)^{\top}
\end{equation*}
and the exponent $t(k) \in \mathbb{N}$ indicates the number of consensus step. 

 To state convergence results, we introduce the averaging of any $np$-dimensional vector $\bold x$:
\begin{align}
\bar{\bf x}=\frac{1}{n}\big(1_n1^{\top}_n\!\!\otimes I_p\big){\bold x}
\end{align}
where $1_n$ is a column vector in $\mathbb{R}^n$ whose entries are all $1$. In particular, we  have
\begin{align}\label{11t}
\bar{{\bf x}}_k=\frac{1}{n}\big(1_n1^{\top}_n\!\!\otimes I_p\big){\bf x}_k~~\mbox{  and  }~~ 
\overline{\nabla  F} (\mathbf{x}_k)=\frac{1}{n}\big(1_n1^{\top}_n\!\!\otimes I_p\big)\nabla F (\mathbf{x}_k).
\end{align}
We note that
$$
\bar{{\bf x}}_k=1_n\otimes\left(\frac{1}{n}\sum_{i=1}^nx_{i,k}\right)=1_n\otimes \bar x_k= (\bar x_{k}^{\top},\bar x_{k}^{\top},\cdots,\bar x_{k}^{\top})^{\top}.
$$
Using these notation for averaged quantities, we derive the following averaged version of (\ref{NEAR DGD}) which plays an important role as a useful intermediate step in the convergence analysis:
\begin{equation}\label{average}
\bar{{\bf x}}_{k+1} = \bar{{\bf x}}_k - \mu \overline{\nabla F}({\bf x}_k),
\end{equation}
where we used the property that $1_n^\top W = 1_n^\top$ coming from the column-stochasticity of the consensus matrix $W$.

Under the assumptions of strong convexity and Lipschitz continuity of the gradient of each function $f_i$, the linear convergence
of (\ref{NEAR DGD}) with $t(k) =k$ is established in \cite{BBKW - Near DGD}, i.e., there exist constants $C>0$, $D>0$ and \textcolor{black}{$\rho\in(0,1)$} such that 
\begin{equation*}
\|{\bf x}_{k} - \bar{{\bf x}}_k \|\leq \beta^{k} D\quad \textrm{and}\quad \|{\bf\bar{x}}_k - {\bf x}_*\| \leq C \rho^k
\end{equation*}
for all $1 \leq i \leq n$ and $k \in \mathbb{N}\cup\{0\}$. 
\textcolor{black}{We recall that $\beta\in[0,1)$ is the spectral norm of $W-\frac{1}{n}1_n1_n^\top$.}

In the current work, we are interested in the case where the strong convexity of the cost functions are missing. \textcolor{black}{Precisely, we obtain the convergence results when either (1)  the cost function is convex or (2) it is given in the form of the composite function of a strongly convex function and a linear function. We note that the objective function of the latter is convex and quasi-strongly convex.}

\subsection{$\bullet$ Main result I: Convergence for convex functions} 
Our first results concern the case where only the convexity (not the strong convexity) 
 is assumed on the cost function $f$. Throughout this paper, $\|\cdot\|$ denotes the usual Euclidean norm. \textcolor{black}{We also recall that $\beta\in[0,1)$ is the spectral norm of $W-\frac{1}{n}1_n1_n^\top$.}
\begin{thm}\label{Main1} Suppose that the cost function $f$ is convex and each local cost $f_i$ is $L_i$-smooth 
for some $L_i>0$. Assume that $\mu \in (0,2/L]$ with $L = \max_{1 \leq i \leq n} L_i$ and the sequence $\{t(k)\}_{k \in \mathbb{N}\cup\{0\}}$ satisfies $\sum_{k=0}^\infty \beta^{t(k)} <\infty$. Then, for any minimizer $x_*\in X^*$ , we have
\begin{equation}\label{eq-1-1}
\|{\bar{\bf x}}_k -{\bf x}_*\|=O(1)\quad \textrm{and}\quad\|\bar{\bold x}_{k+1} -\bold x_{k+1}\| = O(\beta^{t(k)}),\quad k\to\infty,
\end{equation}
where 
$\bold x_*:=1_n\otimes x_* =(x_*^{\top},x_*^{\top},\cdots,x_*^{\top})^{\top}$. 	
\end{thm}

\begin{thm}\label{thm-1-2} Suppose that the cost function $f$ is convex and each local cost $f_i$ is $L_i$-smooth for some $L_i>0$.  Assume that $\mu \in (0, 1/L]$ with $L = \max_{1 \leq i \leq n} L_i$ and the sequence $\{t(k)\}_{k \in \mathbb{N}\cup\{0\}}$ satisfies $\sum_{k=0}^{\infty}\beta^{t(k)}<\infty$. 
	 Then   we have
	 \begin{equation*}
	 f\Bigg(\frac{1}{T} \sum_{k=0}^{T-1}  \bar{x}_k \Bigg) - f^*= O\Big( \frac{1}{T}\Big)
	 \end{equation*}
	 for $T \in \mathbb{N}$.
\end{thm}

In order to obtain the above results, we will find a recursive inequality for the sequence of vectors consisting of $\|\bar{\bf x}_k -{\bf x}_*\|$ and $\|\bar{\bf x}_k - {\bf x}_k \|$. The recursive inequality at step $k$ involves a multiplication with a $2\times 2$ matrix depending on $\beta^{t(k)}$ followed by an addition with a two-dimensional vector depending on $\beta^{t(k)}$. By obtaining a sharp bound on multiplications of those matrices, we first prove the uniform boundedness of the sequence. We then make use of the bound with the convexity of the cost function to derive the convergence result.

\begin{rem}
We note that obtaining estimates \eqref{eq-1-1} is equivalent to obtaining similar estimates for $\|\bar{x}_k -x_*\|$  and $\|x_{i,k} - \bar{x}_k\|$ for $1 \leq i \leq n$ in view of the following relations:
\begin{equation*}
\|\bar{\bf x}_k -{\bf x}_*\| = \sqrt{n} \|\bar{x}_k -x_*\|\quad \textrm{and}\quad \|\bar{\bf x}_k - {\bf x}_k\| = \Big( \sum_{j=1}^n \|\bar{x}_k - x_{j,k}\|^2 \Big)^{1/2}.
\end{equation*}
\end{rem}
\noindent$\bullet$ {\bf Main result 2: Convergence for quasi-strongly convex functions.}
Our second result is concerned with the case where the quasi-strong convexity is assumed on the cost function in place of the strong convexity. 
Especially, we shall consider the case when the quasi-convexity
arises from the composition of a strongly convex function and a rank-deficient matrix, which
appears ubiquitously in optimization problems or machine learning tasks. The following is the definition of the quasi-strongly convex function:

\begin{defn}[\cite{NNG}]
Continuously differentiable function $f$ is called quasi-strongly convex
on set X if there exists a constant  $\kappa_f > 0$ such that 
\[
f^*:=\min_{y\in X}f(y)\geq f(x)+\langle \nabla f(x),[x]-x\rangle+\frac{\kappa_f}{2}\|x-[x]\|^2
\]
for all $x\in X$. Here $[x]$ denotes the projection of $x$ onto the optimal set 
\[
X^*:=\{x\in X: f(x)=f^*\}.
\]
\end{defn}
The projection $[x]$ is well
defined if $X^*$ is closed and convex (see From Theorem 1.5.5 in \cite{FP}). The following composite objective function constitutes a relevant example of quasi-strongly convex functions \cite{MJ}:

\begin{equation}\label{eq-2-21}
f(x)=\textcolor{black}{\sum_{i=1}^n (h_i^\top  x  -y_i)^2 = \|H x-y\|^2,}
\end{equation}
where $x  \in \mathbb{R}^p$, $h_i \in \mathbb{R}^p$, and $y_i \in \mathbb{R}$ and 
\begin{equation*}
H = (h_1  , \cdots, h_n  )^\top\quad \textrm{and}\quad y= (y_1, \cdots, y_n)^\top.
\end{equation*}
Note that the above function  $f$ is strongly convex if $H$ is full ranked, while
it is only quasi-strongly convex if $H$ is not of full rank. In general, if $g: \mathbb{R}^m \rightarrow \mathbb{R}$ is an $\alpha$-strongly convex and $L$-smooth function and $H \in \mathbb{R}^{m \times p}$, then the function $f(x) = g(Hx)$ is quasi-strongly convex.
Now we state our main result for quasi-strongly convex functions.
\begin{thm}\label{Main2} Suppose that  each local cost $f_i$ is $L_i$-smooth for some $L_i>0$ and  the aggregate cost function $f:\mathbb{R}^p\rightarrow \mathbb{R}$ takes the form $f(x) = g(Hx)$ for an $\alpha$-strongly convex and $\mathbf{L}$-smooth function $g:\mathbb{R}^m\rightarrow \mathbb{R}$ with $H \in \mathbb{R}^{m\times p}$.  Assume further that $f$ has a minimizer and
\begin{equation}\label{eq-2-20}
D := \sup_{x  \in X^*} \Big( \sum_{j=1}^n \|\nabla f_j (x )\|^2\Big)^{1/2} < \infty.
\end{equation}
If the stepsize $\eta>0$ and the sequence $\{t(k)\}_{k \in \mathbb{N}\cup \{0\}}$ satisfy
	\begin{equation*}
	\mu  \leq \frac{2C_H}{\mathbf{L}+\alpha}
	\quad \textrm{and}\quad \sum_{k=0}^{\infty} \beta^{t(k)} <\infty,
	\end{equation*}
	  then the discrepancy
	$\|\bold{x}_{k}-{ [{\bar{\bf x}}_k]}\|$ 
    vanishes as $k\rightarrow\infty$, where we adopted the convention $[{\bar{\bf x}}_k]:=1_n\otimes[{\bar{ x}}_k]$. Moreover we have

    \begin{equation*}
\big\|\bold{ x}_{T+1} - [\bar{\bf x}_{T+1}]\big\|
\leq  C\Big( \max_{ \textcolor{black}{\lceil T/2\rceil \leq j \leq T}}\beta^{t(j)}\Big)+ C\Big((1-C_2 \mu)^{T/2}\Big),
\end{equation*}  
where $C_2 = \frac{2\mathbf{L}\alpha c_H}{\mathbf{L}+\alpha}$ and constant $C>0$ is independent of $T$.
 Here $c_H$ denotes a coefficient in the Hoffman inequality  and $C_H = 1/\|H\|_2^2$ with $\|H\|_2: = \sup_{x\in \mathbb{R}^p \setminus \{0\}} \|Hx\|_2 / \|x\|_2$. \textcolor{black}{Also, $\lceil z\rceil$ is the least integer greater than or equal to z.}

	\end{thm}

\begin{rem}The assumption \eqref{eq-2-20} holds if $H$ is the identity matrix since $X^*$ is a point set. In the Lemma \ref{lem-2-13} below, we prove that the assumption \eqref{eq-2-20} is true for natural cases containing the example \eqref{eq-2-21}.
\end{rem} 
\begin{rem}\label{rem-2-8}  Hoffman coefficient  $c_H$ \cite{Hoff,NNG} is defined by the constant that satisfies the following inequality  
		\begin{equation*}
	\|Hx -H[x]\|^2 \geq c_H \|x-[x]\|^2
	\end{equation*}
	for all $x\in \mathbb{R}^m$. More general definitions are available, but this is sufficient for our purpose. The above inequality implies that $c_H \leq 1/C_H =\|H\|_2^2$ since $\|Hx- H[x]\| \leq \|H\|_2 \|x-[x]\|$. 
	\end{rem}
	
	\begin{rem}
 In the above theorem, we note that $f$ is a convex function either. Therefore the optimal set $X^*$ is closed and convex, and so the projection $[x]$ onto $X^*$ is well defined. We remark that a quasi-strongly convex function need not be a convex function in general. 
\end{rem}
In order to prove the above result, we establish a  coercivity estimate for \textcolor{black}{the quasi-strongly convex functions} of the composite form (see Lemma \ref{lem-4-1}). With the help of the coercivity estimate, we will derive a recursive inequality for the sequence of vectors  consisting of $\|\bar{\bf x}_k -[ \bar{\bf x}_k]\|$ and $\|\bar{\bf x}_k - {\bf x}_k\|$. \textcolor{black}{Then, as in the proof of Theorems \ref{Main1} and \ref{thm-1-2}, we find a sharp bound on multiplications of the matrices in the recursive inequality to derive the convergence result.}


It will be an interesting problem to extend this convergence result to general quasi-strongly convex functions. We refer to the recent paper \cite{NNG} where the linear convergence was obtained for gradient descent methods applied to quasi-strongly convex functions.

The results of above theorems are proved by analyzing the growth of the matrix norm appearing in the sequential inequality of the vectors  $\|\bar{\bf x}_k -[ \bar{\bf x}_k]\|$ and $\|\bar{\bf x}_k - {\bf x}_k\|$. We also provide another approach for analyzing the sequential inequality of the vectors, which gives an improved estimate under weak assumptions on $\{t(k)\}$.  In the approach, we first prove that the sequences are uniformly bounded and then use this fact to analyze two types of inequalities for the vectors, separately. We first find a sharp bound of $\|\bar{\bf x}_k -[ \bar{\bf x}_k]\|$, and then use it to obtain a sharp bound of $\|\bar{\bf x}_k - {\bf x}_k\|$.

By the procedure stated above, we aim to obtain refined results for the convergence  of \eqref{NEAR DGD} when the cost $f$ is of the form $f(x) = g(Hx)$ as in Theorem \ref{Main2}.  In the following two theorems, we use the notations
\begin{equation*}
A_0 = \|\bar{\bf x}_0 -[ \bar{\bf x}_0]\|\quad \textrm{and} \quad B_0 =\|\bar{\bf x}_0 - {\bf x}_0\|.
\end{equation*}
In addition, we use the constant $D>0$ given in \eqref{eq-2-20} and a constant $R>0$ defined by
\begin{equation*}
R = \max \Big\{ A_0, ~ B_0 / \gamma,~ \frac{\mu D \beta^J}{\gamma-(\mu L + \gamma (1+\mu L))\beta^J}\Big\},
\end{equation*}
where $\gamma = \frac{1- \sqrt{1-C_2 \mu}}{\mu L}$ and $L = \max_{1 \leq i \leq n} L_i$ with $C_2> 0$ and $L_i >0$ defined in Theorem \ref{Main2}. 

 First we obtain the result when the number $t(k)$ of communications at each step is constant for $k$.
\begin{thm}\label{thm-1-10}Suppose that local costs and the aggregate cost function are given as in Theorem \ref{Main2}. Assume that $t(k) =J \in \mathbb{N}$ for all $k \in \mathbb{N}\cup \{0\}$ and  
\begin{equation*}
\mu \leq \min \Big\{\frac{2C_H}{\mathbf{L}+\alpha},~ \frac{C_2}{C_2 +L(1+\sqrt{2})}\cdot \frac{1-\beta^J}{L\beta^J}\Big\}
\end{equation*}
with constant $C_2 >0$ given in Theorem \ref{Main2} and $L = \max_{1\leq i \leq n} L_i$.
Then, for all $k \geq 0$ we have
\begin{equation}\label{eq-5-24}
\|\bar{\bf x}_{k+1} -[\bar{\bf x}_{k+1}]\| \leq (\sqrt{1-C_2 \mu})^{k} (A_0 + \mu L B_0)   +\frac{\mu L \beta^J}{1-\beta^J}\Big( \frac{\mu  ( 2LR +  D)}{1-\sqrt{1-C_2 \mu}} +  { B_0}\Big)
\end{equation}
and
\begin{equation*}
\|\bar{\bf x}_k - {\bf x}_k\| \leq (  \beta^J)^k B_0 + \frac{\mu \beta^J ( 2LR + D)}{1- \beta^J}.
\end{equation*}
\end{thm}
Next we elaborate the argument used in the above result to obtain a sharp convergence result of \eqref{NEAR DGD} for general non-decreasing sequence $\{t(k)\}$. In particular, we do not assume that \mbox{$\sum_{k=0}^{\infty}\beta^{t(k)}< \infty$.}
\begin{thm}\label{thm-1-11}Suppose that local costs and the aggregate cost function are given as in Theorem \ref{Main2}.  Assume that $\{t(k)\}$ is non-decreasing and we let $J:= t(0)\in \mathbb{N}$. Suppose that 
\begin{equation*}
\mu \leq \min \Big\{\frac{2C_H}{\mathbf{L}+\alpha},~ \frac{C_2}{C_2 + L(1+\sqrt{2})}\cdot \frac{1-\beta^J}{L\beta^J}\Big\}.
\end{equation*}
Then for $k \geq 2$ we have 
\begin{equation}\label{eq-5-40} 
\begin{split}
&\|\bar{\bf x}_{k+1} -[\bar{\bf x}_{k+1}]\|
\\
   &\quad \leq (\sqrt{1-C_2 \mu})^{k} (A_0 + \mu L B_0) 
\\
&\quad \quad   +\frac{\mu L}{1 - \sqrt{1- C_2\mu}} \bigg[ \mu (2LR +D) \beta^{t(\lfloor k/2 \rfloor)} + B_0 \beta^{(\lfloor k/2\rfloor +1)J}\bigg]
\\
&\quad \quad + \mu L (\sqrt{1-C_2 \mu})^{\lfloor k/2 \rfloor} \bigg[ \mu (2LR+D) \sum_{l=0}^{\lfloor k/2 \rfloor -1} \beta^{t(l)} + B_0 \sum_{l=0}^{\lfloor k/2\rfloor -1} \beta^{(l+1)J}\bigg] 
\end{split}
\end{equation}
and
\begin{equation*}
\|\bar{\bf x}_{k+1} - {\bf x}_{k+1}\| \leq \beta^{(k+1)J} B_0 + \mu (2LR +D)\beta^{t(k)}
\end{equation*}
 Here $\lfloor a \rfloor$   for $a>0$ denotes the largest integer not less than $a$.
\end{thm}

Before ending this section, we prove two lemmas regarding the cost functions in the setting of Theorem  \ref{Main2}.
\begin{lem}\label{lem-2-9} Suppose that  each local cost $f_i$ is $L_i$-smooth for some $L_i>0$ and  the aggregate cost function $f:\mathbb{R}^p\rightarrow \mathbb{R}$ takes the form $f(x) = g(Hx)$ for an $\alpha$-strongly convex and $\mathbf{L}$-smooth function $g:\mathbb{R}^m\rightarrow \mathbb{R}$ with $H \in \mathbb{R}^{m\times p}$. Then $L=\max_{1 \leq i \leq n}L_i$ satisfies the inequality $L \geq \alpha \|H\|_{2}^2$. 
	\end{lem}
	\begin{proof}
	The strong convexity of $g$ implies 
	\begin{equation}\label{eq-2-9}
	g(Hy) \geq g(Hx) +\langle  \nabla g(Hx), (Hy-Hx) \rangle + \frac{\alpha}{2} \|Hy - Hx\|^2 \quad \forall~x,y \in \mathbb{R}^m.
	\end{equation}
	Since $f(x) = g(Hx)$ and $\nabla f(x) = H^\top \nabla g(Hx)$, this yields that
	\begin{equation*}
	f(y) \geq f(x) + \langle \nabla f(x), y-x\rangle + \frac{\alpha}{2}\|Hy - Hx\|^2.
	\end{equation*}
	On the other hand, the aggregate cost $f = \frac{1}{n} \sum_{i=1}^n f_i$ is $L$-smooth since each $f_i$ is $L_i$-smooth and $L= \max_{1\leq i \leq n}L_i$. Therefore,
	\begin{equation*}
	f(y) \leq f(x) + \langle \nabla f(x), y-x \rangle + \frac{L}{2} \|y-x\|^2.
	\end{equation*}
	Combining these two inequalities, we find
	\begin{equation}\label{eq-2-6}
	\alpha\|H (x-y)\|^2 \leq L \|x-y\|^2 \quad \forall~x,y \in \mathbb{R}^m,
	\end{equation}
which implies $L \geq \alpha \|H\|_{2}^2$.  
	\end{proof}
	
\begin{lem}\label{lem-2-13} Suppose that  each local cost $f_i$ is $L_i$-smooth for some $L_i>0$ and  the aggregate cost function $f:\mathbb{R}^p\rightarrow \mathbb{R}$ takes the form $f(x) = g(Hx)$ for an $\alpha$-strongly convex and $\mathbf{L}$-smooth function $g:\mathbb{R}^m\rightarrow \mathbb{R}$ with $H \in \mathbb{R}^{m\times p}$.  Assume one of the following statements holds:
\begin{enumerate}
\item The local cost $f_i$ is bounded below for all $1\leq i \leq n$.
\item The local cost $f_i$ satisfies $f_i (x) = f_i (x+v)$ for all  $x \in \mathbb{R}^p$ and $v \in \textrm{Ker}(H)$ and $1 \leq i \leq n$.
\end{enumerate}
Then we have 
\begin{equation}\label{eq-2-13}
 \sup_{x \in X^*} \Big( \sum_{j=1}^n \|\nabla f_j (x )\|^2\Big)^{1/2} < \infty.
\end{equation}
\end{lem}
\begin{proof}
We assume that the first statement holds true. By a standard argument using the smoothness of $f_i$, we have
\begin{equation*}
\min_{y \in \mathbb{R}^p} f_i (y) \leq f_i \Big(x-\frac{1}{L_i} \nabla f_i (x)\Big) \leq f_i (x) - \frac{1}{2L_i} \|\nabla f_i (x)\|^2 \quad \forall~x \in \mathbb{R}^p,
\end{equation*}
which yields
\begin{equation}\label{eq-2-15}
\|\nabla f_i (x) \|^2 \leq 2L_i \Big( f_i (x) - \min_{y \in \mathbb{R}^p} f_i (y)\Big).
\end{equation}
Since $f(x) = \frac{1}{n} \sum_{i=1}^n f_i (x)$  is constant as the optimal value $f^*$ for $x \in X^*$ and each $f_i$ is bounded below, 
\begin{equation*}
\sup_{x \in X^*} f_i (x) < \infty.
\end{equation*}
Combining this with \eqref{eq-2-15} proves \eqref{eq-2-13}.
\

Next we assume the second statement.  Take a point $x_* \in X^*$. Then, using \eqref{eq-2-9} with $f(x) = g(Hx)$ and $\nabla f(x) = H^\top \nabla g(Hx)$, we see that
\begin{equation*}
X^* = \{ x_* + v~:~ v \in \textrm{Ker}(H)\}.
\end{equation*}
For all $v \in \textrm{Ker}(H)$, we have $f_i (x) = f_i (x+v)$ for all $x \in \mathbb{R}^p$. Thus we have $\nabla f_i (x) = \nabla f_i (x+v)$, and so  
\begin{equation*}
\sup_{x \in X^*} \|\nabla f_i (x)\| = \sup_{v \in \textrm{Ker}(H)} \|\nabla f_i (x_* +v)\| = \|\nabla f_i (x_*)\| < \infty.
\end{equation*}
This proves \eqref{eq-2-13}. 
\end{proof}
In the above lemma, we remark that the statement $(2)$ is trivially holds when $H$ is given by an identity matrix since $\textrm{Ker}(H) = \{0\}$. This corresponds to the case that the aggregate cost $f$ is a strongly convex function.


\section{Preliminary results}\label{sec-3}
In this section, we establish a proposition which will play a key role in the proofs of the main theorem. We begin with  recalling the following well-known lemma (see e.g. \cite[Lemma 3.11]{Bu}).
\begin{lem}\label{lem-2-1} Assume that $f: \mathbb{R}^n \rightarrow \mathbb{R}$ is $\alpha$-strongly convex and $L$-smooth. Then, for any $x,y \in \mathbb{R}^n$ we have
\begin{equation*}
\langle \nabla f(x) - \nabla f(y), x-y \rangle \geq \frac{L\alpha}{L+\alpha}\|x-y\|^2 + \frac{1}{\alpha +L} \|\nabla f(x) - \nabla f(y)\|^2.
\end{equation*}
\end{lem}
	In the proofs of the main theorems, we will consider a sequence $\{(A_k, B_k)\}_{k \in \mathbb{N}}$ defined in \eqref{eq-3-10} and derive a recursive estimate for the sequence. Then, to obtain the desired results from the recursive estimate, we will use the following proposition. 

\begin{prop}\label{prop-2-2} 
Let  $\{X_k\}_{k\geq0}$ and $\{Y_k\}_{k\geq0}$ be sequences of non-negative vectors in $\mathbb{R}^2$ and $\{M_k\}_{k\geq0}$ be a sequence of $2\times 2$ matrices. Assume that $\{X_k\}_{k\geq0}$, $\{Y_k\}_{k\geq0}$ and $\{M_k\}_{k\geq0}$ satisfy the following iterative relation:
	\begin{equation*}
	X_{k+1} \leq M_k X_k + Y_k.
	\end{equation*}
	Suppose further that $\{Y_k\}_{k \geq0}$ is a sequence of non-negative vectors such that $\sum_{k=0}^{\infty}\big\| Y_k \big\|< \infty$, and $\{M_k\}_{k\geq0}$ takes the following form:
	 \begin{equation*}
	 M_k = \begin{pmatrix} 1 & r \\ \alpha_k & \alpha_k \end{pmatrix}
	 \end{equation*}
	 for some fixed $r>0$ and a sequence of non-negative real numbers $\{\alpha_k\}_{k\geq0}$ satisfying  $\sum_{k=0}^{\infty} \alpha_k <\infty$. 
 Then, we have
 \begin{enumerate}
 	 	\item The matrix $M_b M_{b-1}\dots M_a$   is uniformly bounded for $0 \leq a < b< \infty$.
 	\item  $\{X_{k}\}_{k\geq0}$ is uniformly bounded.
 	\end{enumerate}
\end{prop} 

\begin{proof}
(1) We take $c = \max\{r,1\}$ and define a weighted norm $\|\cdot\|_{c}$ in $\mathbb{R}^2$ by
\begin{equation*}
\|(x,y)^\top\|_c = |x| + c|y| \quad \textcolor{black}{\mbox{for }x,y\in\mathbb{R}}.
\end{equation*}
For a matrix $M \in \mathbb{R}^{2\times 2}$, we define the  operator norm by
\begin{equation*}
\|M\|_{\mathcal{L}} =\sup_{v \in \mathbb{R}^2 \setminus \{0\}} \frac{\|Mv\|_{c}}{\|v\|_{c}}.
\end{equation*}
We claim that
\begin{equation*}\label{eq-2-3}
\|M_k\|_{\mathcal{L}} \leq 1+ c\alpha_k.
\end{equation*}
To show this, we take an arbitrary $v = (x,y)^\top \in \mathbb{R}^2$ and compute
\begin{equation*}
M_kv=\begin{pmatrix} 1 & r \\ \alpha_k &\alpha_k \end{pmatrix} \begin{pmatrix} x \\ y \end{pmatrix} = \begin{pmatrix} x + ry \\ \alpha_k x + \alpha_k y \end{pmatrix},
\end{equation*}
which has the weighted $L^1$ norm as
\begin{equation*}
\|M_kv\|_{c} = |x+ry| + c |\alpha_k x + \alpha_k y| \leq (1+c\alpha_k)|x| + (r+c \alpha_k )|y|.
\end{equation*}
Since $c = \max\{r,1\}$ we have
\begin{equation*}
r+c \alpha_k \leq c (1+c\alpha_k)
\end{equation*}
so that
\begin{equation*}
\begin{split}
\|M_kv\|_c&\leq(1+c\alpha_k) |x| + (r+c \alpha_k ) |y|
\\
&\leq (1+c\alpha_k) (|x|+c|y|) = (1+c\alpha_k) \|v\|_c.
\end{split}
\end{equation*}
This proves the claim. Using \eqref{eq-2-3}, for any $1\leq a<b<\infty$ we  obtain
\begin{equation*}
\begin{split}
\Big\| M_b M_{b-1}\dots M_a \Big\|_{\mathcal{L}}& \leq \prod_{k=a}^b \|M_k \|_{\mathcal{L}}
\\
& \leq \prod_{k=a}^b (1+c\alpha_k)
\\
& \leq \prod_{k=a}^b e^{c\alpha_k} = \exp \Big(c\sum_{k=a}^b \alpha_k\Big).
\end{split}
\end{equation*}
Combining this with the assumption that $\sum_{k=0}^{\infty} \alpha_k < \infty$, we conclude that $\Big\| M_b M_{b-1}\dots M_a \Big\|_{\mathcal{L}} $ is uniformly bounded for $0\leq a<b <\infty$.

\noindent (2) By induction, the iteration implies that
\begin{equation}\label{eq-2-10}
X_{k+1} \leq \Big( M_k M_{k-1}\dots M_0\Big)X_0 + \sum_{j=0}^{k-1} \Big( M_k M_{k-1}\dots M_{j+1}\Big) Y_j + Y_k.
\end{equation}
This gives the desired uniform boundedness of $X_k$ since $M_b M_{b-1}\dots M_a$ is uniformly bounded and $\sum_{j=0}^{\infty} \big\|Y_j \big\|< \infty$. 
\end{proof}
{\color{black}
We finish this section by proving that the two conditions on the consensus matrix $W$ introduced earlier actually leads to $\beta\in[0,1)$. Its proof was sketched  briefly in \cite{QL}. We provide a more detailed version of the proof for the readers' convenience.
\begin{lem}\label{newlemma}\cite{QL}
Assume that the consensus matrix $W$ satisfies the following:
\begin{itemize}
\item $W$ is a doubly-stochastic $n\times n$ matrix.
\item The directed graph associated to $W$ is strongly connected, and has a self-loop at every vertex.
\end{itemize}
Then the following assertions hold:
\begin{enumerate}
\item The graph associated to $W^\top W$ is connected and aperiodic. (Since $W^\top W$ is symmetric, we may either interpret its associated graph as a strongly connected directed graph, or a connected undirected graph, depending on preference.)
\item The largest eigenvalue of $W^\top W$ is $1$, and it is simple.
\item The spectral norm $\beta$ of $W-\frac{1}{n}1_n1_n^\top$ satisfies $\beta\in[0,1)$.
\end{enumerate}
\end{lem}
\begin{proof}
(1) The assumption that the graph associated to $W$ has a self-loop at every vertex, is equivalent to $w_{ii}>0$ for all $1\leq i\leq n$. Let $w_0:=\frac{1}{2}\min\limits_{1\leq i\leq n}w_{ii}>0$ and $\widetilde W:=W-w_0 I_n$. Then all entries of $\widetilde W$ are nonnegative, and the directed graph associated to $\widetilde W$ is precisely the same as that of $W$. Note that
\[
W^\top W=(\widetilde W+w_0I_n)^\top(\widetilde W+w_0I_n)\geq w_0 \widetilde W,
\]
with the inequality being in entrywise sense. Hence the graph associated to $W^\top W$ contains that associated to $\widetilde W$ (or $W$) as its subgraph, which is indeed strongly connected and aperiodc.

(2) Since $W^\top W$ is symmetric positive semidefinite, we can arrange its eigenvalues as $\lambda_1\geq \lambda_2\geq\dots\geq\lambda_n\geq0$. The Perron-Frobenius theorem combined with (1) yields $1=\lambda_1>\lambda_2$.

(3) Since $W^\top W$ is symmetric, it can be diagonalized by an orthogonal matrix, i.e.,
\[
W^\top W=Q \operatorname{diag}(\lambda_1,\dots,\lambda_n)Q^\top,\quad Q:=[u_1|\dots|u_n],
\]
with $u_i$ being an eigenvector associated to $\lambda_i$, $i=1,\dots,n$. Note that $\lambda_1=1$ and $u_1=\frac{1}{\sqrt{n}}1_n$, because $W$ is doubly-stochastic. Hence
\[
W^\top W=\sum_{i=1}^n \lambda_i u_iu_i^\top=\frac{1}{n}1_n1_n^\top+\sum_{i=2}^n \lambda_i u_iu_i^\top,
\]
which yields
\begin{align*}
\left\|W-\frac{1}{n}1_n1_n^\top\right\|&=\sqrt{\mbox{the largest eigenvalue of }(W-\frac{1}{n}1_n1_n^\top)^\top(W-\frac{1}{n}1_n1_n^\top)}\\
&=\sqrt{\mbox{the largest eigenvalue of }W^\top W-\frac{1}{n}1_n1_n^\top}=\sqrt{\lambda_2}<1,
\end{align*}
where the second equality follows using that $1_n 1_n^\top W = 1_n 1_n^\top$ and $W^\top 1_n 1_n^\top  = 1_n 1_n^\top$ since $W$ is doubly-stochastic. The proof is finished.
\end{proof}

}
\section{Proof of Theorem \ref{Main1} and Theorem \ref{thm-1-2}}\label{sec-4}
In this section, we obtain the convergence estimates of NEAR-DGD$^+$ when the global objective function $f$ is convex. 
\begin{proof}[Proof of Theorem \ref{Main1}]We define
\begin{equation}\label{eq-3-10}
A_k:= \|\bar{\bf x}_k-{\bf x}_*\| = \sqrt{n} \|\bar{x}_k - x_*\|,\quad B_k:=\|{\bf x}_k-\bar{\bf x}_k\|
\end{equation}
and proceed to obtain a recursive estimate for $A_k$ and $B_k$. We first estimate $B_{k+1}$ in terms of $A_k$ and $B_k$. For this, we use the algorithm \eqref{NEAR DGD} along with \eqref{average} and  \eqref{11t} to find
\begin{equation}\label{eq-3-30}
\begin{split}
&\|\bar{{\bf x}}_{k+1} - {\bf x}_{k+1}\|\cr
 & = \|(\bar{{\bf x}}_k - \mu \overline{\nabla F}({\bf x}_k)) - (W^{t(k)}\otimes I_p) ({\bf x}_k - \mu \nabla F({\bf x}_k))\|
\\
& = \left\| \left[\left( \frac{1}{n} 1_n 1_n^\top  - W^{t(k)}\right)\otimes I_p\right] ({\bf x}_k -\mu \nabla F({\bf x}_k) \right\|
\\
&\leq \left\| \left[\left( \frac{1}{n} 1_n 1_n^\top  - W^{t(k)}\right)\otimes I_p\right] {\bf x}_k \right\| + \mu \left\| \left[\left( \frac{1}{n} 1_n 1_n^\top  - W^{t(k)}\right)\otimes I_p\right] \nabla F({\bf x}_k)\right\|.
\end{split}
\end{equation}
Here we can estimate the right hand side using 
\begin{equation}\label{eq-3-31}
\left\{\left( \frac{1}{n} 1_n 1_n^\top  - W^{t(k)}\right)\otimes I_p\right\} \bar{\bf x}_k=0 
	\quad \textrm{and}\quad
	\Big\| \frac{1}{n} 1_n 1_n^\top  - W^{t(k)}\Big\|_2\leq \beta^{t(k)},
	\end{equation}
where the second inequality came from Lemma \ref{newlemma} and the relation
\[
 W^{t(k)}-\frac{1}{n} 1_n 1_n^\top  = \left(W-\frac{1}{n} 1_n 1_n^\top \right)^{t(k)}.
\]
Using \eqref{eq-3-31} in \eqref{eq-3-30} and  the triangle inequality,  we  achieve
\begin{align}\label{consensus estimate}
\begin{split}
\|\bar{{\bf x}}_{k+1} - {\bf x}_{k+1}\|&\leq {}\beta^{t(k)} \|{\bf x}_k - \bar{\bf x}_k\| + {}\mu \beta^{t(k)} \|\nabla F({\bf x}_k)\|\cr
 &\leq {}\beta^{t(k)} \|{\bf x}_k - \bar{\bf x}_k\| + {}\mu \beta^{t(k)} 
 \left\{\|\nabla F({\bf x}_k)-\nabla F(\bar{\bf x}_k)\|\right.\cr
 &\quad\left.+\|\nabla F(\bar{\bf x}_k)-\nabla F({\bf x}_*)\|
 +\|\nabla F({\bf x}_*)\|
 \right\}.\cr
 \end{split}
 \end{align}
We note that the $L_j$-smoothness of $f_j$ yields that for $x, y \in \mathbb{R}^{np}$ we have
\begin{equation*}
\begin{split}
\| \nabla F (x) - \nabla F(y) \| & = \Bigg( \sum_{j=1}^n \|\nabla f_j (x^j) - \nabla f_j (y^j) \|^2 \Bigg)^{1/2} 
\\
&\leq \Bigg(\sum_{j=1}^n \textcolor{black}{L_j^2 \|x^j -y^j\|^2} \Bigg)^{1/2} 
\\
& \leq L \|x-y\|,
\end{split}
\end{equation*}
where we denoted $x= ((x^1)^\top, \cdots, (x^n)^\top)^\top$ and $y= ((y^1)^\top, \cdots, (y^n)^\top)^\top$. Using this in \eqref{consensus estimate} we obtain
 \begin{equation}\label{eq-3-32}
 \begin{split}
 \|\bar{{\bf x}}_{k+1} - {\bf x}_{k+1}\| &\leq {}(1+\mu L)\beta^{t(k)} 
 \|{\bf x}_k-\bar{\bf x}_k \|
 +{}(\mu L)\beta^{t(k)}\|\bar{\bf x}_k-{\bf x}_*\|
 +{}\mu\beta^{t(k)}D
 \\
 & ={}(1+\mu L)\beta^{t(k)} 
 \|{\bf x}_k-\bar{\bf x}_k \|
 + {}(\mu \sqrt{n} L)\beta^{t(k)}\|\bar{x}_k-{x}_*\|
 +{}\mu\beta^{t(k)}D,
\end{split}
\end{equation}
where $D:=\|\nabla F({\bf x}_*)\|$. This gives
\begin{align}\label{estimate A}
B_{k+1} \leq  {}(\mu L)\beta^{t(k)}A_k+{}(1+\mu L)\beta^{t(k)}B_k+{}\mu\beta^{t(k)} D. 
\end{align}
Next we estimate the error $A_{k+1}$ between the average and the optimizer $x_*$ in terms of $A_k$ and $B_k$. Using \eqref{average} and the triangle inequality, we find
\begin{align}\label{main comp}
\begin{split}
\|\bar{{ x}}_{k+1} - { x}_*\| 
 &= \bigg\|\bar{{x}}_k - {x}_* - \frac{\mu}{n} \sum_{j=1}^n \nabla f_j (x_{j,k})\bigg\|\\
& = \bigg\|\bar{{ x}}_k - {x}_* - \frac{\mu}{n} \sum_{j=1}^n \nabla f_j (\bar{x}_k) + \mu \bigg( \frac{1}{n} \sum_{j=1}^n \nabla f_j (\bar{x}_k)- \frac{1}{n} \sum_{j=1}^n \nabla f_j (x_{j,k})\bigg)\bigg\|\cr
&\leq \|\bar{{  x}}_k - {  x}_* - \mu \nabla f (\bar{{ x}}_k) \|+ \mu \bigg\|\frac{1}{n} \sum_{j=1}^n \nabla f_j (\bar{x}_k)- \frac{1}{n} \sum_{j=1}^n \nabla f_j (x_{j,k})\bigg\|.
\end{split}
\end{align}
We make use of $\nabla f(x_*)=0$ to compute
\begin{equation}\label{eq-3-1}
\begin{split}
&\|\bar{{x}}_k- {x}_*-\mu{\nabla f}(\bar{{x}}_k)\|^2\\
&\quad= \|\bar{{ x}}_k-{ x}_*\|^2
-2\mu\langle  \bar{{x}}_k-{x}_*,\nabla f(\bar{{ x}}_k)\rangle
+\mu^2\|\nabla f(\bar{{ x}}_k)\|^2\\
&\quad=\|\bar{{x}}_k-{x}_*\|^2
-2\mu\langle  \bar{{x}}_k-{x}_*, {\nabla f}(\bar{{x}}_k)
- {\nabla f}({{x}}_*)\rangle
+\mu^2\|{\nabla f}(\bar{{x}}_k)-{\nabla f}({x}_*)\|^2.
\end{split}
\end{equation}
From Lemma \ref{lem-2-1}, we find
\begin{equation}\label{eq-3-2}
\langle \bar{x}_k -{x}_*,  {\nabla f}(\bar{x}_k)- {\nabla f}({ x}_*)\rangle \geq\frac{1}{L}\| \textcolor{black}{{\nabla f}(\bar{ x}_k)}-{\nabla f}({x}_*)\|^2,
\end{equation}
and apply it to \eqref{eq-3-1} to obtain
\begin{align*}
\begin{split}
& \|\bar{{x}}_k-{x}_*-\mu{\nabla f}(\bar{{x}}_k)\|^2\\
&\quad\leq\|\bar{{x}}_k-{x}_*\|^2+\mu\left(\mu-\frac{2}{L}\right)\|{\nabla f}(\bar{{x}}_k)-{\nabla f}({{x}}_*)\|^2\cr
&\quad\leq \|\bar{{x}}_k-{x}_*\|^2
\end{split}
\end{align*}
for $\mu \in (0, 2/L]$.
This gives us the estimate
\begin{align}\label{eq-3-6}
&\|\bar{{x}}_k-{x}_*-\mu {\nabla f}(\bar{{ x}}_k)\|
\leq \|\bar{{x}}_k-{x}_*\|.
\end{align}
On the other hand, we use the $L_j$-smoothness of $f_j$ and the triangle inequality to deduce
\begin{align}\label{insert1}
\begin{split} 
 \bigg\|\frac{1}{n} \sum_{j=1}^n \nabla f_j (\bar{x}_k)- \frac{1}{n} \sum_{j=1}^n \nabla f_j (x_{j,k})\bigg\| & \leq \frac{1}{n} \sum_{j=1}^n L_j \|\bar{x}_k - x_{j,k}\|
 \\
&\leq \frac{L}{\sqrt{n}} \|\bar{{\bf x}}_k - {\bf x}_k\|.
\end{split}
\end{align}
Combining \eqref{main comp},  \eqref{eq-3-6} and \eqref{insert1}, we get
\begin{align}\label{line for A}
\begin{split}
\|\bar{{x}}_{k+1}- {{x}}_*\|
\leq \|\bar{{x}}_k- {{x}}_*\|+ \frac{\mu L}{\sqrt{n}}\|{\bf x}_k-\bar{{\bf x}}_k\|.
\end{split}
\end{align}
From \eqref{estimate A} and \eqref{line for A}, we get the following iterative relation:
\begin{equation}\label{eq-3-50}
\left(\begin{array}{c} A_{k+1}\cr B_{k+1}\end{array}\right) \leq
\left(\begin{array}{cc} 1 &  \mu L\cr {}(\mu L) \beta^{t(k)}& {} (1+\mu L)\beta^{t(k)}\end{array}\right)
\left(\begin{array}{c} A_{k}\cr B_{k}\end{array}\right)
+
{}\beta^{t(k)}\left(\begin{array}{c} 0\cr \mu D\end{array}\right),
\end{equation}
\textcolor{black}{
which clearly implies
\begin{align*}
\left(\begin{array}{c} A_{k+1}\cr B_{k+1}\end{array}\right) \leq
\left(\begin{array}{cc} 1 &  \mu L\cr {}(1+\mu L) \beta^{t(k)}& {} (1+\mu L)\beta^{t(k)}\end{array}\right)
\left(\begin{array}{c} A_{k}\cr B_{k}\end{array}\right)
+
{}\beta^{t(k)}\left(\begin{array}{c} 0\cr \mu D\end{array}\right). 
\end{align*}
}
Now we apply Proposition \ref{prop-2-2} to conclude that there exists $R>0$ such that
\begin{equation*}
\|\bar{\bf x}_k -\bold x_*\| \leq R \quad \textrm{and}\quad \|\bar{\bold x}_k -\bold x_k \| \leq R
\end{equation*}
for all $k \in \mathbb{N}$. By applying this bound to \eqref{estimate A}, we obtain
\begin{equation*}
\|\bar{\bold x}_{k+1} -\bold x_{k+1}\| \leq {}\beta^{t(k)} \Big[ (1+2\mu L) R + \mu D \Big].
\end{equation*}
This completes the proof.
\end{proof}

\begin{proof}[Proof of Theorem \ref{thm-1-2}]
Applying \eqref{eq-3-6} and 
\eqref{insert1} to \eqref{main comp}, we find
\begin{equation}\label{eq-3-12}
\begin{split}
\|\bar{x}_{k+1} -x_*\|^2 & \leq \| \bar{x}_k - x_* - \mu {\nabla f}(\bar{x}_k) \|^2
\\
&\quad + \frac{2 \mu L }{\sqrt{n}}\|\bar{x}_k - x_* \| \|\bar{\bold x}_k -\bold x_k\| + \frac{\mu^2 L^2}{n} \|\bar{\bold x}_k -\bold x_k\|^2.
\end{split}
\end{equation}
We use {$\nabla f(x_*)=0$} and apply \eqref{eq-3-2} to obtain  
\begin{equation*}
	\begin{split}
		- 2\mu \langle \bar{{x}}_k -{x}_*,~ {\nabla f}(\bar{{ x}}_k) \rangle		& = - 2\mu \langle \bar{{x}}_k -{x}_*,~  {\nabla f}(\bar{{x}}_k)-  {\nabla f}( {x}_*) \rangle 
		 \\
		& \leq- \mu \langle \bar{{x}}_k -{x}_*,~ {\nabla f}(\bar{{x}}_k) \rangle - \frac{\mu}{L} \|\nabla
f(\bar{{x}}_k)\|^2
	\\
		& \leq \mu\left\{f({{x}}_*)-f(\bar{{x}}_k)\right\}  - \frac{\mu}{L} \|\nabla
 f(\bar{{x}}_k)\|^2,
	\end{split}
\end{equation*}
where we used 
the convexity of $f$ in the second inequality. Using this we find
\begin{equation}\label{insert2}
\begin{split}
&\|\bar{{x}}_k - {x}_* - \mu {\nabla f}(\bar{{x}}_k) \|^2 
\\
&\qquad = \|\bar{{x}}_k -{x}_*\|^2 - 2\mu \langle \bar{{x}}_k -{x}_*,~ {\nabla f}(\bar{{x}}_k) \rangle + \mu^2 \|{\nabla f}(\bar{{x}}_k)\|^2
\\
&\qquad \leq \|\bar{{x}}_k - {x}_*\|^2 + 
\mu\left\{f({{x}}_*)-f(\bar{{x}}_k)\right\}
 + (\mu^2 - \mu/L) \| {\nabla f}(\bar{{x}}_k)\|^2.
\end{split}
\end{equation}
Inserting (\ref{insert2}) into \eqref{eq-3-12}, we obtain
\begin{equation}\label{above}
\begin{split}
\|\bar{{x}}_{k+1} - {x}_*\|^2 & \leq \|\bar{{x}}_k - {x}_*\|^2 +
\mu\left\{f({{x}}_*)-f(\bar{{x}}_k)\right\}
 + (\mu^2 - \mu /L) \|{\nabla f}(\bar{{x}}_k)\|^2
\\
&\quad +\frac{2\mu L}{\sqrt{n}}\|\bar{{x}}_k -{x}_*\| \|\bar{{\bf x}}_k - {\bf x}_k \| + \frac{\mu^2 L^2}{n} \|\bar{{\bf x}}_k - {\bf x}_k \|^2
\\
&\leq \|\bar{{x}}_k - {x}_*\|^2 + \mu\left\{f({{x}}_*)-f(\bar{{x}}_k)\right\}
\\
&\quad + \frac{2\mu L}{\sqrt{n}} \|\bar{{x}}_k -{x}_*\| \|\bar{{\bf x}}_k -{\bf x}_k\| + \frac{\mu^2 L^2}{n}\|\bar{{\bf x}}_k - {\bf x}_k\|^2
\end{split}
\end{equation}
provided $\mu\in (0, 1/L]$.
%
Combining   \eqref{eq-1-1} with (\ref{above}), we obtain
\begin{equation*}
\|\bar{{x}}_{k+1} - {x}_*\|^2 \leq \|\bar{{x}}_k - {x}_*\|^2 + \mu\left\{f({{x}}_*)-f(\bar{{x}}_k)\right\} + 
C\beta^{t({k-1})}  + 
C\beta^{2t(k-1)}
\end{equation*}
for some constant $C>0$.
Summing this up over $k={1},\cdots T-1$, we get 
\begin{equation*}
\begin{split}
&\mu\sum_{k=1}^{T-1} \left\{f(\bar{{x}}_k)-f({{x}}_*)\right\} + \|\bar{{x}}_{T} -{x}_*\|^2 
\\
&\quad \leq \|\bar{{x}}_{1} -{x}_*\|^2 +C \sum_{k=0}^{{T-2}} \beta^{t(k)}  + C\sum_{k=0}^{{T-2}}  \beta^{2t(k)}.
\end{split}
\end{equation*}
Dividing both sides by $T$ and applying the convexity of $f$ to the left-hand side, we obtain
\begin{equation*}
\begin{split}
& f\Bigg( \frac{1}{T}\sum_{k=0}^{T-1}  \bar{x}_k \Bigg) -f(\bar{{x}}_*) \\
&\leq \frac{1}{\mu T}\Bigg[{\mu\left\{f(\bar{{x}}_0)-f(\bar{{x}}_*)\right\}}+\|\bar{{ x}}_{1} -{x}_*\|^2 + C\sum_{k=0}^{{T-1}} \beta^{t(k)}  + C\sum_{k=0}^{{T-1}} \beta^{2t(k)}\Bigg],
\end{split}
\end{equation*}
which completes the proof.
\end{proof}
\section{Proof of Theorem \ref{Main2}}\label{sec-5}
In this section, we establish the convergence result for the NEAR-DGD$^{+}$ when the global objective function belongs to a class of \textcolor{black}{convex and quasi-strongly convex functions.} Namely, we consider the class of convex functions that arise from the composition of strongly convex function and possibly \textcolor{black}{rank-deficient} matrix. The following lemma provides a coercivity estimate for those functions which will be an essential tool for the convergence analysis.
\begin{lem}\label{lem-4-1} Assume the function $f:\mathbb{R}^p\rightarrow \mathbb{R}$ takes the form $f(x) = g(Hx)$ for $H\in\mathbb R^{m\times p}$ and an $\alpha$-strongly convex and $L$-smooth function $g:\mathbb{R}^{m}\rightarrow \mathbb{R}$. Then for some $c_H>0$, we have
	\begin{equation*}
	\begin{split}
	&\langle \nabla f(x) - \nabla f([x]), x-[x]\rangle
	\\
	& \qquad\geq \frac{{L}\alpha c_H}{L + \alpha}\|x -[x]\|^2 + \frac{C_H}{L + \alpha}\|\nabla f(x) - \nabla f([x])\|^2,
	\end{split}
	\end{equation*}
	where $c_H$ is the coefficient of the Hoffman inequality \cite{NNG, WL}, and $C_H =1/\|H \|_2$ with $\|H \|_2 := \sup_{x \in \mathbb{R}^p \setminus \{0\}} \frac{\|H x\|_2}{\|x\|_2}$.
\end{lem}
\begin{proof}
	By applying Lemma \ref{lem-2-1} with the fact that $g$ is strongly convex and smooth, we find
	\begin{equation*}
	\begin{split}
	&\langle\nabla f(x) - \nabla f(y) ,\, x-y \rangle
	\\
	&\qquad = \langle H^\top  \nabla g(Hx) - H^\top  \nabla g(Hy) , \,x-y \rangle
	\\
	&\qquad =\langle \nabla g(Hx) - \nabla g(Hy) , \,Hx - Hy  \rangle
	\\
	&\qquad  \geq \frac{{L}\alpha}{L + \alpha}\|Hx -Hy\|^2 + \frac{1}{L + \alpha}\|\nabla g(Hx) - \nabla g(Hy)\|^2.
	\end{split}
	\end{equation*}
	With the choice of $y=[x]$, we get
	\begin{equation}\label{put1}
	\begin{split}
	&\langle\nabla f(x) - \nabla f([x]) ,\,  x-[x] \rangle
	\\
	&\qquad \geq \frac{{L}\alpha}{L + \alpha}\|Hx -H[x]\|^2 + \frac{1}{L + \alpha}\|\nabla g(Hx) - \nabla g(H[x])\|^2,
	\end{split}
	\end{equation}
	Then, we have from the Hoffman inequality \cite{NNG, WL} that
	\begin{equation}\label{put2}
	\|Hx - H[x]\|^2 \geq c_H \|x-[x]\|^2
	\end{equation}
	and from the definition of $f$ that
	\begin{equation}\label{put3}
	\begin{split}
	\|\nabla g (Hx) - \nabla g(H[x])\| &\geq C_H \|A^\top  \nabla g(Hx) -H^\top  \nabla g(H[x])\| 
	\\
	&= C_H \|\nabla f(x) - \nabla f ([x])\|.
	\end{split}
	\end{equation}
	Inserting \eqref{put2}, \eqref{put3} into \eqref{put1}, we get the desired result.	
\end{proof}
We now turn to the proof of Theorem \ref{Main2}.
\subsection{Proof of the Theorem \ref{Main2}}:
{In this proof, we denote any positive constant that is independent of the number of iterations $k\in\mathbb N\cup\{0\}$, by $C>0$.}  First we use \eqref{eq-2-6} to compute
\begin{align}\label{eq-4-5}
\begin{split}
\|\bar{{x}}_{k+1}-{ [\bar{x}_{k}]}\|&=\bigg\|\bar{{x}}_k-{ [\bar{ x}_{k}]}-\frac{\mu}{n} \sum_{j=1}^n \nabla f_j (x_{j,k}) \bigg\|\cr
&=\bigg\|\bar{{x}}_k-{ [\bar{ x}_{k}]}-\frac{\mu}{n} \sum_{j=1}^n \nabla f_j (\bar{x}_k)
+\frac{\mu}{n} \sum_{j=1}^n \Big(\nabla f_j (\bar{x}_k)-\nabla f_j (x_{j,k})\Big) \bigg\|\cr
&\leq\Big\|\bar{{x}}_k-{ [\bar{ x}_{k}]}-{\mu} \nabla f (\bar{x}_k)\Big\|
+\frac{\mu L}{n} \sum_{j=1}^n\|\bar{x}_k - x_{j,k}\|\cr
&\leq\Big\|\bar{{x}}_k-{ [\bar{ x}_{k}]}-{\mu} \nabla f (\bar{x}_k)\Big\|
+\frac{\mu L}{\sqrt{n}}\|\bar{{\bf x}}_k-{\bf x}_k\|.\\
\end{split}
\end{align} 
For the first term of the last line, we use ${\nabla f}({ [\bar{ x}_{k}]})=0$ to compute
\begin{align*}
&\Big\|\bar{{x}}_k-{ [\bar{ x}_{k}]}-{\mu} \nabla f (\bar{x}_k)\Big\|^2\\
&\quad\textcolor{black}{=} \|\bar{{x}}_k-{ [\bar{x}_{k}]}\|^2
-2\mu\langle  \bar{{x}}_k-{ [\bar{x}_{k}]}, {\nabla f}(\bar{{x}}_k)\rangle
+\mu^2\|{\nabla f}(\bar{{x}}_k)\|^2\\
&\quad=\|\bar{{x}}_k-{ [\bar{x}_{k}]}\|^2
-2\mu\langle  \bar{{x}}_k-{ [\bar{x}_{k}]},\nabla f(\bar{{x}}_k)
-{\nabla f}({ [\bar{ x}_{k}]})\rangle
\\
&\quad\quad+\mu^2\|{\nabla f}(\bar{{x}}_k)-{\nabla f}({ [\bar{x}_{k}]})\|^2.
\end{align*}
Then we apply Lemma \ref{lem-4-1} and use $0 < \mu \leq \frac{2C_H}{\mathbf{L} + \alpha}$ to bound the last line by
\begin{align*}
\begin{split}
& \|\bar{{x}}_k-{ [\bar{x}_{k}]}\|^2
-2\mu\left(\frac{\mathbf{L}\alpha c_H}{\mathbf{L} + \alpha}\|\bar{x}_k -{ [\bar{x}_{k}]}\|^2 + \frac{C_H}{\mathbf{L}+ \alpha}\|{\nabla f}(\bar{{x}}_k) -{\nabla f}({ [\bar{x}_{k}]})\|^2\right)\\
&\quad+\mu^2\|{\nabla f}(\bar{{x}}_k)-{\nabla f}({ [\bar{x}_{k}]})\|^2\\
&\quad \leq
\left(1-2\mu\frac{{\mathbf{L}}\alpha c_H}{\mathbf{L}+ \alpha} \right)\|\bar{{x}}_k-{ [\bar{x}_{k}]}\|^2.
\end{split}
\end{align*}
Therefore, we have 
\begin{align}\label{we have from 3}
&\|\bar{{x}}_k-{ [\bar{x}_{k}]}-\mu{\nabla f}(\bar{{x}}_k)\|
\leq\sqrt{1-C_2\mu}\|\bar{{x}}_k-{ [\bar{x}_{k}]}\|,
\end{align}
where $C_2 = \frac{2{ \mathbf{L}}\alpha c_H}{\mathbf{L} + \alpha}$. 
Inserting \eqref{we have from 3} into \eqref{eq-4-5}, we get
\begin{align*}\label{we have 4}
\begin{split}
\|\bar{{x}}_{k+1}-{ [\bar{x}_{k}]}\|
\leq \sqrt{1-C_2\mu}\|\bar{{x}}_k-{[\bar{ x}_k]}\|+\frac{\mu L}{\sqrt{n}} \|{\bf x}_k-\bar{{\bf x}}_k\|.
\end{split}
\end{align*}
This implies
\begin{equation}\label{eq-4-20}
\|\bar{{x}}_{k+1}-{[\bar{ x}_{k+1}]}\|
\leq \sqrt{1-C_2\mu}\|\bar{{x}}_k-{ [\bar{ x}_k]}\|+\frac{\mu L}{\sqrt{n}} \|{\bf x}_k-\bar{{\bf x}}_k\|
\end{equation}
since $[x]$ is the projection of $x$ onto the minimizing set $X^*$.
On the other hand, we recall from \eqref{eq-3-32} that
\begin{equation}\label{eq-4-21}
\begin{split}
\|{\bf x}_{k+1}-\bar{{\bf x}}_{k+1} \|
 \leq {(1+\mu L)}\beta^{t(k)} 
\|{\bf x}_k-\bar{\bf x}_k\| +{\mu \sqrt{n} L}\beta^{t(k)}\|\bar{x}_k-[\bar{x}_k]\|
+{\mu D}\beta^{t(k)},
\end{split}
\end{equation}
where $D>0$ is now defined as in \eqref{eq-2-20}.
If we put
\begin{equation}\label{eq-4-50}
A_k=\sqrt{n}\|\bar{{ x}}_{k}-{   [\bar{ x}_k]}\|,\quad B_k=\|{\bf x}_{k}-\bar{{\bf x}}_{k} \|,
\end{equation}
 we get the following system of  difference inequalities:
\begin{align}\label{eq-4-1}
\left(\begin{array}{c} A_{k+1}\cr B_{k+1}\end{array}\right) \leq
\left(\begin{array}{cc} \sqrt{1-C_2\mu} & \mu L \cr {C_0}\beta^{t(k)} & {C_0}\beta^{t(k)}\end{array}\right)
\left(\begin{array}{c} A_{k}\cr B_{k}\end{array}\right)
 +
{}\beta^{t(k)}\left(\begin{array}{c} 0\cr \mu D\end{array}\right), 
\end{align}
where $C_0 = (1+\mu L)$.
We let $q = \sqrt{1-C_2\mu}$ and define
\begin{equation*}
Z_k = \begin{pmatrix} A_k \\ B_k \end{pmatrix},\quad Y_k={ }\beta^{t(k)}\left(\begin{array}{c} 0\cr \mu D\end{array}\right),\quad \textrm{and}\quad M_k = \left(\begin{array}{cc} \sqrt{1-C_2\mu} & \mu L\cr {C_0}\beta^{t(k)} & {C_0}\beta^{t(k)}\end{array}\right).
\end{equation*}
Noting that
\begin{equation*}
M_k = q \left(\begin{array}{cc} 1 & (\mu L)/q\cr ({C_0}/q)\beta^{t(k)} & ({C_0}/q)\beta^{t(k)}\end{array}\right) 
\end{equation*}
and applying Proposition \ref{prop-2-2}, we see that 
\begin{equation}\label{eq-4-10}
\|M_bM_{b-1}\dots M_a \|_1\leq {C}  \,q^{b-a+1}
\end{equation}
for all $1\leq a \leq b<\infty$. 
 We now expand \eqref{eq-4-1} as in \eqref{eq-2-10} and apply \eqref{eq-4-10} with the triangle inequality to deduce
\begin{equation*}
\begin{split}
\|Z_{T+1}\| & \leq \left\|\Big( M_T M_{T-1}\dots M_1 \Big) Z_0 + \sum_{j=0}^{T-1} \Big( M_T M_{T-1}\dots M_{j+1} \Big) Y_j + Y_T\right\|
\\
&\leq Cq^{T}\| Z_0\| + C\sum_{j=0}^{T-1} q^{T-j} \|Y_j\| + \|Y_T\|
\\
& \leq Cq^{T}\|Z_0\|+ C \sum_{j=0}^{T} q^{T-j} \beta^{t(j)}.
\end{split}
\end{equation*}
We estimate this further as
\begin{equation*}
\begin{split}
\|Z_{T+1}\|& \leq Cq^{T}\|Z_0\| + C\sum_{j=\lceil T/2\rceil}^{T} q^{T-j} \beta^{t(j)} + C \sum_{j=0}^{\lceil T/2\rceil -1} q^{T-j} \beta^{t(j)} 
\\
&\leq Cq^{T} \|Z_0\|+C\Big( \max_{\lceil T/2\rceil \leq j \leq T}\beta^{t(j)}\Big) \sum_{k=0}^{\infty} q^k + C \sum_{j=0}^{\lceil T/2\rceil -1} q^{T-j} 
\\
&\leq Cq^{T} \|Z_0\| +C\Big( \max_{ \lceil T/2 \rceil \leq j \leq T}\beta^{t(j)}\Big)\frac{1}{1-q} + \frac{Cq^{ T/2}}{1-q}  .
\end{split}
\end{equation*}
Since $\|Z_k \| = \sqrt{n\|\bar{x}_k - [\bar{x}_k]\|^2 + \|{\bf x}_k - \bar{\bf x}_k\|^2}$ for $k \in \mathbb{N}$, the above estimate gives the desired estimate. \textcolor{black}{The proof is finished.}

\section{\textcolor{black}{Convergence results for fixed steps of consensus}}\label{sec-6}

In this section, we provide another approach to analyze the sequential inequality \eqref{eq-4-1}. We first prove the uniform boundedness of the sequences from the inequality and then use it to analyze each of the two sequential inequalities in  \eqref{eq-4-1}. This approach results in the improved convergence estimates of \eqref{NEAR DGD} under weaker condition on $t(k)$ containing the case that $t(k)$ is constant.

We recall from \eqref{eq-4-50} the sequences defined by
\[
A_k=\sqrt{n}\|\bar{{ x}}_{k}-{   [\bar{ x}_k]}\|,\quad B_k=\|{\bf x}_{k}-\bar{{\bf x}}_{k} \|.
\]
In the following lemma, we show that the sequences $\{A_k\}_{k \geq 0}$ and $\{B_k\}_{k \geq 0}$ are uniformly bounded. 
\begin{lem}\label{lem-5-1} Fix  $J \in \mathbb{N}$. Assume that $t(k) \geq J$ for all $k \in \mathbb{N}\cup \{0\}$ and  
\begin{equation*}
\mu < \min \Big\{\frac{2C_H}{\mathbf{L}+\alpha},~ \frac{C_2}{C_2 +L (1+\sqrt{2}) }\cdot \frac{1-\beta^J}{L\beta^J}\Big\},
\end{equation*}
where $C_2 = \frac{2\mathbf{L}\alpha c_H}{\mathbf{L}+\alpha}$. 
We take $\gamma = \frac{1- \sqrt{1-C_2 \mu}}{\mu L}$ and 
\begin{equation}\label{eq-5-30}
R = \max \Big\{ A_0, ~ B_0 / \gamma,~ \frac{\mu D \beta^J}{\gamma-(\gamma +(1+\gamma)\mu L)\beta^J}\Big\}.
\end{equation}
Then we have
\begin{equation}\label{eq-5-20}
A_k \leq R, \quad B_k \leq \gamma R
\end{equation}
for all $k \in \mathbb{N} \cup \{0\}$.
\end{lem}

\begin{rem}In the definition of $R$ of the above lemma, we need the condition $(\gamma+(1+ \gamma) \mu L ) \beta^J < \gamma$, which is written as
\begin{equation}\label{eq-5-13}
\mu < \frac{\gamma(1-\beta^J)}{(1+\gamma) L \beta^J}.
\end{equation} 
Noting that $\gamma = \frac{C_2}{L (1+\sqrt{1+C_2 \mu}) }$and $C_2 \mu <1$, we have
\begin{equation*}
 \gamma>\frac{C_2}{L(1+\sqrt{2}) }.
\end{equation*} 
It yields that $\frac{\gamma}{1+\gamma} \geq \frac{C_2}{C_2 + L(1 +\sqrt{2}) }$, and so the following condition
\begin{equation*}
\mu < \frac{C_2}{C_2 + L(1+\sqrt{2}) } \cdot \frac{1-\beta^J}{L\beta^J}
\end{equation*}
is a sufficient condition for \eqref{eq-5-13}.
\end{rem}
\begin{proof}[Proof of Lemma \ref{lem-5-1}]
We argue by an induction. Combining $t(k) \geq J$ with \eqref{eq-4-20} and \eqref{eq-4-21}, we have the following estimates
\begin{equation}\label{eq-5-1}
\begin{split}
A_{k+1} & \leq \sqrt{1-C_2 \mu} A_k + \mu L B_k
\\
B_{k+1} & \leq \mu L\beta^{J} A_k + (1+\mu L) \beta^{J} B_k + \mu D \beta^J.
\end{split}
\end{equation}
Note that \eqref{eq-5-20} holds for $k=0$ by the definition of $R$. Assume that it holds that
\begin{equation}\label{eq-5-2}
A_k \leq R,\quad B_k \leq \gamma R
\end{equation}
for a fixed value $k \in \mathbb{N}\cup \{0\}$. Combining this with the first estimate of \eqref{eq-5-1} yields
\begin{equation*}
A_{k+1} \leq (\sqrt{1-C_2 \mu} + \gamma \mu L) R = R.
\end{equation*}
We also use the second estimate of \eqref{eq-5-1} with \eqref{eq-5-2} to deduce
\begin{equation*}
\begin{split}
B_{k+1} &\leq \mu  L\beta^J R +   (1+\mu L) \beta^J \gamma R + \mu D \beta^J
\\
& =  (\mu L+ \gamma (1+\mu L)) \beta^J R + \mu D \beta^J
\\
& \leq \gamma R,
\end{split}
\end{equation*}
where the second inequality holds due to $R \geq  \frac{\mu  D \beta^J}{\gamma-  (\mu L+ \gamma (1+\mu L))\beta^J}$. This completes the inductive argument, and so the proof is completed.
\end{proof}
\begin{rem}
Notice that using the definition of $C_2$ along with the inequality $c_H \leq \|H\|_2^2$ (see Remark \ref{rem-2-8}) and $L \geq \alpha \|H\|_2^2$ of Lemma \ref{lem-2-9}, we have
\begin{equation}\label{eq-6-6}
\begin{split}
\gamma & = \frac{C_2}{L(1+\sqrt{1+C_2 \mu})   } 
\\
& \leq \frac{C_2}{2 {L}} = \frac{2\mathbf{L} \alpha c_H}{2L (\mathbf{L}+\alpha)  }
\\
& \leq \frac{\alpha \|H\|_2^2\mathbf{L}}{L(\mathbf{L}+\alpha)} 
\\
&\leq \frac{\mathbf{L}}{\mathbf{L}+\alpha}<1.
\end{split}
\end{equation}
\end{rem}
Using the uniform boundedness of the above lemma, we proceed to analyze the sequential estimates \eqref{eq-4-20} and \eqref{eq-4-21} further to derive a sharp bound on $A_k$ and $B_k$. For this we shall use an elementary lemma given in the below.
\begin{lem}\label{lem-5-2}
Let $\alpha \in (0,1)$ and $\beta >1$. Suppose that a sequence $\{v_k\}_{k=0}^{\infty}\subset\mathbb{R}$ satisfies 
\begin{equation}\label{eq-5-10}
v_{k+1} \leq \alpha v_k + \beta
\end{equation}
for all $k \in \mathbb{N} \cup \{0\}$. Then we have
\begin{equation*}
v_k \leq \alpha^k v_0 + \frac{\beta}{1-\alpha}.
\end{equation*}
\end{lem}
\begin{proof}
Using \eqref{eq-5-10} inductively, we find
\begin{equation*}
v_k \leq \alpha^k v_0 + \beta\sum_{j=0}^{k-1} \alpha^j < \alpha^k v_0 + \frac{\beta}{1-\alpha}.
\end{equation*}
The proof is done.
\end{proof}
We derive a convergence for \eqref{NEAR DGD} when the number of the consensus steps $t(k)$ is fixed in the following result, which corresponds to Theorem \ref{thm-1-10}.
\begin{proof}[Proof of Theorem  \ref{thm-1-10}]
Using the bound of Lemma \ref{lem-5-1} with \eqref{eq-5-1} and \eqref{eq-6-6}, we find
\begin{equation*}
B_{k+1}   \leq    \beta^{J} B_k + \mu \beta^J ( 2L   R  + D).
\end{equation*} 
By applying the above lemma, we deduce that
\begin{equation*}
B_k \leq ( \beta^J)^k B_0 + \frac{\mu \beta^J ( 2LR + D)}{1- \beta^J}.
\end{equation*}
Inserting this estimate into the first inequality of \eqref{eq-5-1} yields
\begin{equation*}
A_{k+1} \leq \sqrt{1-C_2 \mu} A_k + \mu^2 L \beta^J \frac{(2 L R + D)}{1- \beta^J} + \mu L( \beta^J)^k B_0.
\end{equation*}
We consider two sequences $v_k$ and $w_k$ such that
\begin{equation}\label{eq-5-11}
v_{k+1} = \sqrt{1-C_2 \mu} \, v_k +  \mu^2 L \beta^J \frac{( 2L R + D)}{1- \beta^J}
\end{equation}
for $k \geq 1$ with $v_1 = A_1$, and
\begin{equation}\label{eq-5-12}
w_{k+1} = \sqrt{1-C_2 \mu} \, w_k + \mu L ( \beta^J)^k B_0
\end{equation}
for $k \geq 1$ and $w_1 =0$. Then it is easy to see that $A_k \leq v_k + w_k$ for all $k \geq 1$.  
Applying Lemma \ref{lem-5-2} to \eqref{eq-5-11} yields 
\begin{equation*}
v_{k} \leq (\sqrt{1-C_2 \mu})^{k-1} A_1 + \frac{\mu^2}{1-\sqrt{1-C_2 \mu}} \frac{L\beta^J (2 L R +  D)}{1-  \beta^J}.
\end{equation*}
Also, noting that \eqref{eq-5-12} implies that $w_{k+1} \leq w_k +\mu L ( \beta^J)^k B_0$, we deduce that
\begin{equation*}
w_k \leq \sum_{l=1}^{k-1} \mu L ( \beta^J)^l B_0 \leq \frac{\mu L B_0}{1-   \beta^J}  \beta^J.
\end{equation*}
Combining the above estimates with the fact that $A_k \leq v_k + w_k$, we get
\begin{equation*}
A_k \leq (\sqrt{1-C_2 \mu})^{k-1} A_1 + \frac{\mu^2}{1-\sqrt{1-C_2 \mu}} \frac{L\beta^J ( 2LR +  D)}{1-  \beta^J} +  \frac{\mu LB_0}{1-   \beta^J}  \beta^J.
\end{equation*}
Using that $A_1 \leq A_0 + \mu LB_0$ from \eqref{eq-5-1}, we obtain
\begin{equation}\label{eq-5-24} 
A_{k}  \leq (\sqrt{1-C_2 \mu})^{k-1} (A_0 + \mu LB_0)   +\frac{\mu L \beta^J}{1-\beta^J}\Big( \frac{\mu  ( 2LR +  D)}{1-\sqrt{1-C_2 \mu}} +  { B_0}\Big)
\end{equation}
The proof is done.
\end{proof}
\begin{rem}
We notice that \eqref{eq-5-24} gives a bound of the form
\begin{equation*}
A_k \leq (\sqrt{1-C_2 \mu})^{k-1}(A_0 + \mu L B_0) + O \Big( \frac{\mu L \beta^J}{1-\beta^J}\Big)
\end{equation*}
using that
\begin{equation*}
\frac{\mu^2}{1 - \sqrt{1-C_2 \mu}}  = \frac{\mu (1 + \sqrt{1-C_2 \mu})}{C_2} \leq \frac{2\mu}{C_2}.
\end{equation*}
\end{rem}
We utilize the above argument to handle general step size $\{\beta (k)\}$ to get the following result, which corresponds to Theorem \ref{thm-1-11}.
\begin{proof}[Proof of Theorem \ref{thm-1-11}]
Using \eqref{eq-4-21}, we have
\begin{equation}\label{eq-5-23}
\begin{split}
B_{k+1} &\leq \mu L \beta^{t(k)} A_k + (1+\mu L)\beta^{t(k)} B_k + \mu D \beta^{t(k)}
\\
&= \beta^{t(k)} B_k +\mu \beta^{t(k)} (LA_k +LB_k +D)
\\
&\leq \beta^{t(k)} B_k + \mu (2LR +D) \beta^{t(k)},
\end{split}
\end{equation}
where we used Lemma \ref{lem-5-1} and \eqref{eq-6-6} in the second inequality.
Since $t(k) \geq J$, we have
\begin{equation*} 
B_{k+1}  \leq \beta^{J} B_k + \mu (2LR +D) \beta^{J}. 
\end{equation*}
By applying Lemma \ref{lem-5-2} to the above inequality, we get
\begin{equation*}
B_k \leq \beta^{kJ} B_0 + \frac{\mu (2LR+D) \beta^J}{1-\beta^J}.
\end{equation*}
Inserting this estimate into \eqref{eq-5-23}, we get
\begin{equation*}
\begin{split}
B_{k+1}& \leq \beta^{kJ+t(k)} B_0 + \mu (2LR+D) \beta^{t(k)} \Big( \frac{\beta^J}{1-\beta^J} + 1\Big)
\\
&= \beta^{(k+1)J} B_0 + \frac{\mu (2LR +D)\beta^{t(k)}}{1-\beta^{J}}.
\end{split}
\end{equation*}
Combining this with the first inequality of \eqref{eq-5-1} yields 
\begin{equation*}
A_{k+1} \leq \sqrt{1-C_2 \mu} A_k + \frac{\mu^2 L(2LR +D) \beta^{t(k-1)}}{{1-\beta^{J}}} + \mu L B_0 \beta^{kJ}
\end{equation*}
for $k \geq 1$.
Using this iteratively gives
\begin{equation}\label{eq-6-22}
\begin{split}
A_{k+1} &\leq (\sqrt{1-C_2 \mu})^{k} A_1 +\frac{\mu^2 L (2LR+D)}{1-\beta^{J}} \sum_{l=0}^{k-1} (\sqrt{1-C_2 \mu})^{k-1-l} \beta^{t(l)}
\\
&\quad  + \mu L B_0 \sum_{l=0}^{k-1} (\sqrt{1-C_2 \mu})^{k-1-l} \beta^{(l+1)J}.
\end{split}
\end{equation}
Splitting the summation, we estimate 
\begin{equation*}
\begin{split}
&\sum_{l=0}^{k-1} (\sqrt{1-C_2 \mu})^{k-1-l} \beta^{t(l)}
\\
 & = \sum_{l=0}^{\lfloor k/2\rfloor-1} (\sqrt{1-C_2 \mu})^{k-1-l} \beta^{t(l)} + \sum_{l=\lfloor k/2\rfloor}^{k-1} (\sqrt{1-C_2 \mu})^{k-1-l} \beta^{t(l)}
 \\
&\leq (\sqrt{1-C_2 \mu})^{\lfloor k/2\rfloor} \sum_{l=0}^{\lfloor k/2\rfloor-1} \beta^{t(l)}  + \frac{\beta^{t(\lfloor k/2\rfloor)}}{1- \sqrt{1-C_2 \mu}}.
\end{split}
\end{equation*}
Similarly,
\begin{equation*}
\begin{split}
&\sum_{l=0}^{k-1} (\sqrt{1-C_2 \mu})^{k-1-l} \beta^{(l+1)J}
\\
& = \sum_{l=0}^{\lfloor k/2\rfloor-1} (\sqrt{1-C_2 \mu})^{k-1-l} \beta^{(l+1)J} + \sum_{l=\lfloor k/2\rfloor}^{k-1}(\sqrt{1-C_2 \mu})^{k-1-l} \beta^{(l+1)J}
\\
&\leq (\sqrt{1-C_2 \mu})^{\lfloor k/2\rfloor} \sum_{l=0}^{\lfloor k/2\rfloor-1} \beta^{(l+1)J}  + \frac{\beta^{(\lfloor k/2\rfloor+1)J}}{1-\sqrt{1-C_2 \mu}}.
\end{split}
\end{equation*}
Combining the above two estimates in \eqref{eq-6-22} we have
\begin{equation*}
\begin{split}
A_{k+1}& \leq (\sqrt{1-C_2 \mu})^{k} A_1 
\\
&\quad   +\frac{\mu^2 L (2LR+D)}{1-\beta^{J}} \bigg( (\sqrt{1-C_2 \mu})^{\lfloor k/2\rfloor} \sum_{l=0}^{\lfloor k/2\rfloor-1} \beta^{t(l)} + \frac{\beta^{t(\lfloor k/2\rfloor)}}{1- \sqrt{1-C_2 \mu}}\bigg)
\\
&\quad + \mu L B_0 \bigg((\sqrt{1-C_2 \mu})^{\lfloor k/2\rfloor} \sum_{l=0}^{\lfloor k/2\rfloor-1} \beta^{(l+1)J}  + \frac{\beta^{(\lfloor k/2\rfloor+1)J}}{1-\sqrt{1-C_2 \mu}}\bigg).
\end{split}
\end{equation*}
Using that $A_1 \leq A_0 + \mu LB_0$ from \eqref{eq-5-1}, we obtain
\begin{equation*}
\begin{split}
A_{k+1}& \leq (\sqrt{1-C_2 \mu})^{k} (A_0 + \mu L B_0) 
\\
&\quad   +\frac{\mu L}{1 - \sqrt{1- C_2\mu}} \bigg[ \mu (2LR +D) \beta^{t(\lfloor k/2 \rfloor)} + B_0 \beta^{(\lfloor k/2\rfloor +1)J}\bigg]
\\
&\quad + \mu L(\sqrt{1-C_2 \mu})^{\lfloor k/2 \rfloor} \bigg[ \mu (2LR+D) \sum_{l=0}^{\lfloor k/2 \rfloor -1} \beta^{t(l)} + B_0 \sum_{l=0}^{\lfloor k/2\rfloor -1} \beta^{(l+1)J}\bigg].
\end{split}
\end{equation*}
The proof is done.
\end{proof}


\section{Numerical Results}\label{sec-7}
In this section, we present numerical results which validate the convergence results obtained in this paper. 

\subsection{Regression Problem} We perform the numerical test for the regression problem 
\begin{equation*}
f(x)=\sum_{i=1}^n \| h_i^\top  x  -y_i \|^2 = \|H x-y\|^2,
\end{equation*}
where $x  \in \mathbb{R}^p$, $h_i \in \mathbb{R}^{p\times s}$, and $y_i \in \mathbb{R}^s$ and 
\begin{equation*}
H = (h_1, \cdots h_n )^\top\quad \textrm{and}\quad y= (y_1^\top, \cdots, y_n^\top)^\top.
\end{equation*}
This function $f$ is quasi-strongly convex, and it is strongly convex if and only if $\textrm{ker}(H) =\{0\}$. We recall that $\textrm{ker}(H)=\{0\}$ is equivalent to $H^\top H$ being invertible. Thus we may determine whether $f$ is strongly convex or not, by computing $\det (H^\top H)$. 

We present the simulation for the case when $f$ is quasi-strongly convex but not strongly convex ($\textrm{ker}(H) \neq \{0\}$). In this case we let $X^*$ be the set of minimizers of $f$, and for $x \in \mathbb{R}^p$ we denote by $[x]$ the projection of $x$ onto the set $X^*$. We measure $\|\bar{x}_k - [\bar{x}_k]\|$ and plot this quantity in terms of the following measures (see Figure 1 below)
\begin{enumerate}
\item Number of iterations;
\item Cumulative cost of communication and computation, i.e., sum of the costs corresponding to the iteration numbers $0,1,2, \dots, k$.\end{enumerate}
Here the cost of communication and computation is defined by 
\begin{center}
\textrm{Cost = $\#$ Communications $\times c_c + \#$ Computations $\times c_g$.}
\end{center}
The above cost introduced in \cite{BBKW - Near DGD} concerns both the cost of communications and that of gradient computations with suitable weight $c_c >0$ and $c_g >0$ reflecting the background of the problem.

\begin{figure}[htbp]
\includegraphics[height=4cm, width=7cm]{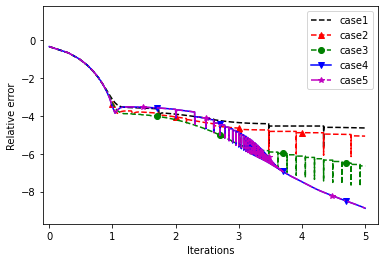}  \includegraphics[height=4cm, width=7cm]{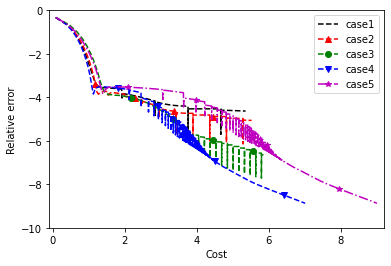} 
\caption{Performance of NEAR-DGD$^{+}$ for the regression problem with respect to: (1) Log$_{10}$ of Number of iterations, (2) Log$_{10}$ of Cumulative cost of communication and computation.}
\label{fig1}
\vspace{-0.3cm}
\end{figure}
We set  the dimension $p=50$, the number of agents $n=8$, and $s=7$. For each \textcolor{black}{$1\leq i \leq n$}, we set the $p \times s$ matrix $h_i$ by sampling each entry according to the uniform distribution on the interval $[0,1]$. Note that the $ns\times p$ matrix $H$ must have a nonzero kernel since $ns=56>50=p$. Then, we set $y_i = h_i x_i$ with a vector $x_i \in \mathbb{R}^p$. \textcolor{black}{Each element of the initial data $x_{i,0}$ is sampled accoring to the uniform distribution on $[0,1]$.}  
Given the initial data, we perform the algorithm \eqref{NEAR DGD} with the following choices: 
\medskip
\begin{enumerate}
\item[Case 1]: $t(k) = \lfloor0.5 * \log (k+1)\rfloor+1$. 
\item[Case 2]: $t(k) = \lfloor  \log (k+1)\rfloor+1$. 
\item[Case 3]: $t(k) = \lfloor 3* \log (k+1)\rfloor +1$.
\item[Case 4]: $t(k) = \lfloor k/100\rfloor+1$.
\item[Case 5]: $t(k) = k+1$,
\end{enumerate}
where  {$\lfloor a\rfloor$ denotes the greatest integer no larger than $a$.}
The left plot in Figure \ref{fig1} represents the graph of $\log_{10} (\|\bar{x}_k - [\bar{x}_k]\|/\|\bar{x}_0 - [\bar{x}_0]\|) $ with respect to the number of iterations in the logarithmic scale. The right plot in Figure \ref{fig1} is the graph of $\log_{10} (\|\bar{x}_k - [\bar{x}_k]\|/\|\bar{x}_0 - [\bar{x}_0]\|)$ with respect to the cost in the logarithmic scale with $c_c =0.2$ and $c_g =1$.
The learning rate is chosen to be $\mu =0.1$.  The graph $\mathcal{G}$ is constructed as a 3-regular
graph topology where each node is connected to its 3 immediate
neighbors.   In the graph `Iteration-Error' we observe that the convergence rate is fastest for Case 4 and Case 5. On the other hand, in the graph `Cost-Error' we see that Case 3 and Case 4 are most efficient. 
\textcolor{black}{
The second largest singular value $\beta$ of the connectivity matrix $W$ for the network of agents is computed as $\beta \simeq 0.577$. For Cases $1$--$3$, the dominating term  in the right hand side of the convergence estimate \eqref{eq-5-40} of Theorem \ref{thm-1-11} is 
\begin{equation}\label{eq-6-2}
\beta^{t(\lfloor k/2\rfloor)} \simeq \beta^{q \log (k/2+1)} = (k/2+1)^{q \log \beta} = (k/2+1)^{-q |\log \beta|},
\end{equation}
where $q\in \{0.5, 1, 3\}$. The log of these bounds are observed in the first graph of \mbox{Figure 1}. For Cases 4--5, the dominating term in the right hand side of \eqref{eq-5-40} is 
\begin{equation}
 (\sqrt{1-C_2 \mu})^{\lfloor k/2\rfloor} \sum_{l=0}^{\lfloor k/2\rfloor-1} \beta^{t(l)} \simeq  (\sqrt{1-C_2 \mu})^{\lfloor k/2\rfloor}.  
 \end{equation}
This bound is also observed in the graph $(1)$ of Figure 1. In fact, for these cases, the graph exhibits two phases of exponential decays. This is related with the bound \eqref{eq-5-40} involving two exponentially decaying terms 
\begin{equation}\label{eq-6-1}
\sigma_1 (\sqrt{1-C_2 \mu})^{\lfloor k/2\rfloor} + \sigma_2 \beta^{t(\lfloor k/2\rfloor)},
\end{equation}
where $\sigma_1$ and $\sigma_2$ are suitable positive constants. If $(\sqrt{1-C_2 \mu})> \beta \simeq 0.577$ and $\sigma_1 \ll \sigma_2$, then the graph of the log value of  \eqref{eq-6-1}  with increasing $k \geq 1$ involves two phases: the decay induced by $\sigma_2 \beta^{t(\lfloor k/2\rfloor)}$ for $1\leq k \leq k_0$ for some value $k_0 \in \mathbb{N}$,  and the decay induced by $\sigma_1 (\sqrt{1-C_2 \mu})^{\lfloor k/2\rfloor}$ for $k >k_0$. }

\textcolor{black}{We also perform the above numerical test with different choices of $c_c$ and $c_g$ for the Cost. We test the following two values
\begin{equation}
(c_g, c_c) = \{(1,0.02),  \, (0.02, 1) \}.
\end{equation}
The result shows that Cases 3--4 exhibit robust efficiency for any choices of $(c_g, c_c)$ while the order of efficiency of Cases  1 and 5 depend on the choices of $(c_g,c_c)$.}
\begin{figure}[htbp]
\includegraphics[height=4cm, width=7cm]{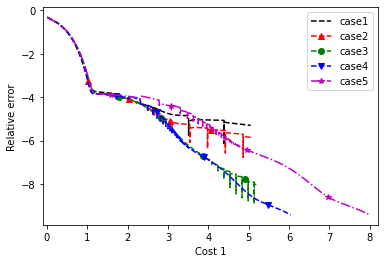}  \includegraphics[height=4cm, width=7cm]{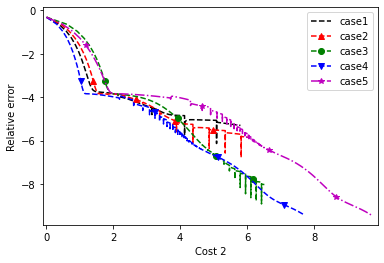} 
\caption{Performance of NEAR-DGD$^{+}$ for the regression problem with respect to different costs. Cost 1 (resp., Cost 2) denotes the Log$_{10}$ of cumulative cost of communication and computation with $(c_g, c_c)$ given by $(1, 0.02)$ (resp., $(0.02,1)$).}
\label{fig1}
\vspace{-0.3cm}
\end{figure}
\medskip


\subsection{Convex Problem}

We consider the convex (but not quasi-strongly convex) problem posed in \cite{QL} given by
\begin{equation}\label{eq-6-3}
f(x) = \frac{1}{n}\sum_{i=1}^n f_i (x) = \frac{1}{n} \sum_{i=1}^n (u(x) + b_i x), 
\end{equation}
where $b_i$ is randomly chosen to satisfy $\sum_{i} b_i =0$, and $u(x) = \frac{1}{4} x^4$ for $|x| \leq 1$, and $u(x) = |x| - \frac{3}{4}$ for $|x| >1$. \textcolor{black}{Again, we consider the same  seven cases of $t(k)$ as in the previous simulation.}
The left and the right plots of Figure \ref{fig2} measure  the summation of the agents' regret 
\begin{equation*}
\log_{10}\Big(\frac{1}{n}\sum_{i=1}^n \Big(f (x_i) -f (0)\Big)\Big),
\end{equation*}
 where $0$ is the minimizer of the function $f$, with respect to the number of iterations and the cost with $c_c =0.2$ and $c_g =1$, respectively. The learning rate is chosen to be $\mu =0.5$. The graph $\mathcal{G}$ is  a 3-regular topology. In the graph `Iterations-Regret' we see that \mbox{Cases 2--5} are fastest in convergence. The function cost decreases consistently also for Case 1. In the graph `Cost-Regret' it is observed that Case 2 is most efficient, and the performance of Case 3 is better than Cases 4--5. From \eqref{eq-6-2}, we see that the condition $\sum_{k=1}^{\infty} \beta^{t(k)} < \infty$ is satisfied for Cases 3--5 and not satisfied for Cases 1--2. However the numerical result shows that the Case 2 exhibits a similar decay to Case 3. This might suggest   that the bound obtained in Theorem \ref{thm-1-2} could be improved for the convex cost function \eqref{eq-6-3} using its property $f(x) \sim |x|^4$ near the zero.
\begin{figure}[htbp]
\includegraphics[height=4cm, width=7cm]{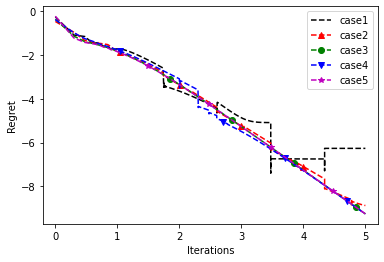}  \includegraphics[height=4cm, width=7cm]{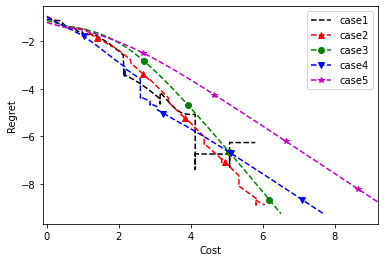}\caption{Performance of NEAR-DGD$^{+}$ for the convex problem with respect to: (1) Log$_{10}$ of Number of iterations, (2) Log$_{10}$ of Cumulative cost of communication and computation.}
\label{fig2}
\vspace{-0.3cm}
\end{figure}

\section{Conclusion}
In this paper, we studied the convergence property of the NEAR-DGD$^+$ (Nested Exact Alternating Recursion Distributed Gradient Descent) method  when the strong convexity assumption is missing. First, we obtained the $O(1/T)$ convergence result when the cost function is convex and smooth. Secondly, we established the linear convergence for a class of  quasi-strongly convex functions which are given as  composition of a degenerate linear mapping and a strongly convex function. We also provided the numerical results supporting the theoretical convergence results. Extending our result to general quasi-strongly convex functions will be an interesting future work. 





\section*{Acknowledgments}
We are grateful to the reviewers for their various comments and suggestions which improved the manuscript. The work of W. Choi was supported by the National Research Foundation of Korea NRF-
2016R1A5A1008055. The work of Seok-Bae Yun was supported by Samsung Science and Technology Foundation under Project Number SSTF-BA1801-02.

\end{document}